\documentclass[reqno]{amsart}

\usepackage{a4wide}
\usepackage{color}
\usepackage{mathrsfs}
\usepackage{mathtools}
\usepackage{amsmath}
\usepackage{amssymb}
\numberwithin{equation}{section}
\usepackage[colorlinks,citecolor=green,linkcolor=red]{hyperref}
\usepackage{tikz-cd}

\usepackage[latin1]{inputenc}

\newcommand{\N}{\mathbb{N}}
\newcommand{\R}{\mathbb{R}}
\newcommand{\sfd}{{\sf d}}
\renewcommand{\d}{{\mathrm d}}
\newcommand{\X}{{\rm X}}
\newcommand{\LIP}{{\rm LIP}}
\newcommand{\lip}{{\rm lip}}
\newcommand{\limi}{\varliminf}
\newcommand{\lims}{\varlimsup}

\newcommand{\fr}{\penalty-20\null\hfill\(\blacksquare\)}

\newtheorem{theorem}{Theorem}[section]
\newtheorem{corollary}[theorem]{Corollary}
\newtheorem{lemma}[theorem]{Lemma}
\newtheorem{proposition}[theorem]{Proposition}
\newtheorem{definition}[theorem]{Definition}

\newtheorem{remark}[theorem]{Remark}

\makeatletter
\@namedef{subjclassname@2020}{%
  \textup{2020} Mathematics Subject Classification}
\makeatother

\title{A note on 
BV and 1-Sobolev functions on the weighted Euclidean space}
\author{Maria Stella Gelli}
\address{Dipartimento di Matematica\\
         Universit\`{a} di Pisa\\
         Largo Bruno Pontecorvo, 5\\
         56127 Pisa\\
         Italy}
\email{maria.stella.gelli@unipi.it}
\author{Danka Lu\v ci\' c}
\address{Dipartimento di Matematica\\
         Universit\`{a} di Pisa\\
         Largo Bruno Pontecorvo, 5\\
         56127 Pisa\\
         Italy}
\email{danka.lucic@dm.unipi.it}

\begin{document}
\date{\today} 
\begin{abstract}
In the setting of the Euclidean space equipped with an arbitrary Radon measure, 
we prove the equivalence between several notions of function of bounded variation
present in the literature. We also study the relation between various definitions of 
\(1\)-Sobolev function.
\end{abstract}
\maketitle
\section{Introduction}
In the setting of the Euclidean space \(\R^d\) equipped with an arbitrary Radon
measure \(\mu\geq 0\) (hereafter referred to as the \emph{weighted Euclidean space}) 
the first notion of function of bounded variation 
(BV, for short) was introduced in the late nineties, 
proposed by Bellettini, Bouchitt\' e, and Fragal\`{a} in \cite{BBF}. 
The approach in there follows the ideas developed in \cite{BBS},
where the Sobolev space 
\(W^{1,p}\) with \(p>1\) has been introduced. 
It is based on a notion of \emph{space tangent to the measure} \(\mu\) and the 
related concept of \(\mu\)-\emph{tangential gradient}, that we shall discuss below.
The study of functional spaces in the weighted Euclidean space setting
is motivated by numerous applications in 
different kinds of variational problems; e.g.\ shape optimization 
\cite{BouchBut,BolBouch, BL},
optimal transport problems with gradient penalization \cite{Louet},
homogenization \cite{Zhikov, Mandallena}.

In the last twenty years, both Sobolev and BV calculus have been extensively studied 
also in a more general setting, that of \emph{metric measure spaces} (namely, 
complete and separable metric spaces endowed with a boundedly-finite Borel measure),
see for example \cite{MIRANDA03, ADM, DMPhD, Sobolev_Pekka}. 
The first instance of the definition of BV function appeared in \cite{MIRANDA03} 
by Miranda, where a relaxation-type approach has been adopted.
Ten years later  it was followed by the definitions by 
Ambrosio and Di Marino in \cite{ADM} and by Di Marino in \cite{DM},
where a thorough 
study of all the approaches has been performed and where
it was also proven that all of them are equivalent. 
All the results from \cite{ADM} and \cite{DM} are collected in Di Marino's 
PhD thesis \cite{DMPhD}, 
to which we will often refer to. 
\smallskip

In the main result of this paper (given in Theorem \ref{thm:equiv_BV}) we will
prove that the notion of BV function proposed in \cite{BBF}, that is tailored 
for the Euclidean setting, coincides with several (equivalent) notions of BV function
coming from the framework of metric measure spaces \cite{DMPhD}. This note
comes as a natural
follow-up to the paper \cite{LPR}, where the equivalence between 
different notions of Sobolev spaces \(W^{1,p}\) with \(p>1\) has been proven.
\smallskip

Let us now briefly explain the main ideas that lie behind the definition of 
BV and \(W^{1,1}\) functions proposed in \cite{BBF}.
The objects that play a key role in this approach are 
bounded vector fields having bounded distributional divergence.
In the sequel, the space of such vector fields will be denoted by 
\({\rm D_\infty}({\rm div}_\mu)\).
Their role is, in a sense, two-fold: the space \({\rm BV}(\R^d, \mu)\) of 
functions of bounded variation is defined 
as the space of those \(1\)-integrable (with respect to \(\mu\)) functions 
\(f\in L^1_\mu(\R^d)\) 
such that the quantity
\[
\|D_\mu f\|\coloneqq \sup\left\{\int_{\R^d}f\,{\rm div}_\mu(v)\,\d \mu:\, 
v\in {\rm D}_{\infty}({\rm div}_\mu), |v|\leq 1\, \mu\text{-a.e.}\right\},
\]
referred to as the \emph{total variation} of \(f\), is finite. On the other hand, 
one can show that there exists a unique (up to \(\mu\)-a.e.\ equality) minimal 
subbundle of \(\R^d\) -- denoted \(\{T_\mu(x)\}_{x\in \R^d}\) -- such that for every 
\(v\in {\rm D}_{\infty}({\rm div}_\mu)\) it holds
\(v(x)\in T_\mu(x)\) for \(\mu\)-a.e.\ \(x\in \R^d\).
It then permits to give the notion of \emph{tangential gradient} \(\nabla_\mu f\)
of a compactly-supported smooth function,
by simply setting 
\[\nabla_\mu f(x)\coloneqq {\rm pr}_{T_\mu(x)}(\nabla f(x)),\quad 
f\in C^{\infty}_c(\R^d).\]
It has been proven in \cite{BBF} that the space \({\rm BV}(\R^d, \mu)\) can be 
equivalently characterized as the domain of finiteness 
of the relaxation (in the strong \(L^1_\mu(\R^d)\)-topology) of the functional
associating to every \(f\in C^{\infty}_c(\R^d)\) the quantity 
\(\int_{\R^d}|\nabla_\mu f|\,\d\mu\)
and set to be \(+\infty\) elsewhere in \(L^1_\mu(\R^d)\).
Moreover,
it holds that 
\begin{equation}\label{eq:intro_tot_var_tangential}
\|D_\mu f\|=\inf\,\limi_{n\to \infty}\int_{\R^d}|\nabla_\mu f_n|\,
\d\mu,\quad f\in {\rm BV}(\R^d,\mu),
\end{equation}
where the infimum is taken among all \((f_n)_n\subseteq C^{\infty}_c(\R^d)\)
converging strongly in \(L^1_\mu(\R^d)\) to \(f\). 

The Sobolev space \(W^{1,1}(\R^d, \mu)\)
is defined (as in the case \(p>1\) in \cite{BBS}) as the completion of 
\(C^\infty_c(\R^d)\) with respect to the norm
\[
\|f\|_{W^{1,1}(\R^d, \mu)}= \|f\|_{L^1_\mu(\R^d)}+\|\nabla_\mu f\|_{L^1_\mu(\R^d;\R^d)}.
\]
Such defined space is indeed a space of functions (not just an abstract Banach space), 
due to the closability property of the tangential gradient operator, which therefore
extends to the whole space \(W^{1,1}(\R^d, \mu)\).
These properties of \(\nabla_\mu\), that have been stated in \cite{BBF}, are 
proven in  Lemma \ref{lem:grad_mu_closed}, Corollary \ref{cor:W11_space_fcts}, and
Proposition \ref{prop:ext_tg_grad} for the sake of completeness.
\smallskip

Looking from the point of view of the metric measure space theory, there are 
two approaches that will be relevant for the purposes of 
the present paper. 
\smallskip

{\color{blue}(i)} The first one is a variant of the relaxation-type approach from
\cite{MIRANDA03}, given in \cite{DMPhD}. This approach involves,
due to the lack (in general) of a smooth structure of the underlying space,
locally Lipschitz functions. The
role of \(|\nabla_\mu f|\) in the relaxation argument above is here played by 
the asymptotic Lipschitz constant (denoted hereafter by \(\lip_a (f)\) for any 
\(f\) Lipschitz; see \eqref{eq:lip_a} for its definition).
Hence, when we stick to the 
specific case of the weighted Euclidean space, 
the space of BV functions \({\rm BV}_{\rm Lip}(\R^d, \mu)\) is defined 
as the set of those \(f\in L^1_\mu(\R^d)\) for which the quantity
\[
\|D_\mu f\|_{\rm Lip}\coloneqq \inf\,\limi_{n\to \infty}
 \int_{\R^d}{\rm\lip}_a(f_n)\,\d\mu 
\]
is finite.  The infimum above is taken among all sequences \((f_n)_{n}\) 
of locally Lipschitz functions 
converging to \(f\) strongly in \(L^1_\mu(\R^d)\). Localizing the above procedure, 
one can associate to each \(f\in {\rm BV}_{\rm Lip}(\R^d, \mu)\)
its \emph{total variation measure} \(|D_\mu f|_{\rm Lip}\),
which on any open set \(\Omega\subseteq \R^d\) reads as 
\[
|D_\mu f|_{\rm Lip}(\Omega)\coloneqq \inf\,\limi_{n\to \infty}
 \int_{\Omega}{\rm\lip}_a(f_n)\,\d\mu.
\]
The corresponding definition of the Sobolev space \(W^{1,1}_{\rm Lip}(\R^d, \mu)\)
is as follows: \(f\in L^1_\mu(\R^d)\) belongs to \(W^{1,1}_{\rm Lip}(\R^d, \mu)\)
if there exists a sequence \((f_n)_n\) of compactly-supported Lipschitz 
functions converging to \(f\) strongly in \(L^1_\mu(\R^d)\) and such that 
\(({\rm lip}_a(f_n))_n\) is weakly convergent in \(L^1_\mu(\R^d)\).

\smallskip

{\color{blue}(ii)} Another definition of BV function we will 
consider was proposed in \cite{DM}
and is based on the notion of \emph{bounded derivation} \(\mathbf b\)
admitting  \emph{bounded divergence} \({\rm div}(\mathbf b)\).
Such a derivation can be thought of as a linear map
acting on boundedly-supported 
Lipschitz functions and having values in the space of 
essentially bounded functions. 
Also, this derivation enjoys a suitable Leibniz rule and a locality property.
We refer to 
Subsection \ref{ssec:BV_der} for the definition, in the specific case of
the weighted Euclidean space, of the above-described space of
derivations, that we denote by \({\rm Der}_b(\R^d, \mu)\).
With this notion at disposal, the space of BV functions 
\({\rm BV}_{\rm Der}(\R^d,\mu)\)
is defined as the space of those \(f\in L^1_\mu(\R^d)\) for which there 
exists a continuous, linear 
(also with respect to the multiplication by Lipschitz functions) operator 
\(
L_f\colon {\rm Der}_b(\R^d, \mu)\to \mathscr M(\R^d),
\)
such that
\[
L_f({\bf b})(\R^d)=-\int_{\R^d}f\,{\rm div}(\mathbf b)\,\d\mu\quad
\text{ for all }\mathbf b\in {\rm Der}_b(\R^d, \mu).
\]
Here we denote by \(\mathscr M(\R^d)\) the space of all finite
signed Borel measures on \(\R^d\). The total variation associated to
a BV function \(f\in {\rm BV}_{\rm Der}(\R^d, \mu)\) is given by the quantity 
\[
\|D_\mu f\|_{\rm Der}\coloneqq \sup\left\{\int_{\R^d}f\,{\rm div}(\mathbf b)\,\d \mu:\, 
\mathbf b\in {\rm Der}_{b}(\R^d, \mu), |\mathbf b|\leq 1\, \mu\text{-a.e.}\right\}.
\]
Similarly, the total variation measure associated with 
\(f\) is given on any \(\Omega\subseteq \R^d\) open by
\[
|D_\mu f|_{\rm Der}(\Omega)
\coloneqq \sup\left\{\int_{\Omega}f\,{\rm div}(\mathbf b)\,\d \mu:\, 
\mathbf b\in {\rm Der}_{b}(\R^d, \mu), |\mathbf b|\leq 1\, 
\mu\text{-a.e.\ and }{\rm supp}(\mathbf b)\Subset \Omega\right\}.
\]
It follows from \cite{DMPhD} that 
\begin{equation}\label{eq:intro_DM_eqBV}
{\rm BV}_{\rm Lip}(\R^d, \mu)={\rm BV}_{\rm Der}(\R^d,\mu),\quad 
|D_\mu f|_{\rm Lip}=|D_\mu f|_{\rm Der}\quad 
\text{ for every }f\in {\rm BV}_{\rm Lip}(\R^d, \mu).
\end{equation}

\medskip

The present paper provides the following results:
\smallskip

{\color{blue} (I)} \({\rm BV}_{\rm Lip}(\R^d, \mu)={\rm BV}_{C^\infty}
(\R^d, \mu)\): As one might expect, we show that in the setting of the 
weighted Euclidean space, 
smooth functions are enough for the approximation in the relaxation process that 
leads to the definition of the space \({\rm BV}_{\rm Lip}(\R^d, \mu)\) described in 
{\color{blue}(i)}.  
The same holds also for the total variation measure, 
namely \(|D_\mu f|_{\rm Lip}=|D_\mu f|_{C^\infty}\) (see Theorem
\ref{thm:Lip_Cinfty}).
Recalling that for every \(f\in C^{\infty}(\R^d)\) it holds that 
\({\rm lip}_a(f)=|\nabla f|\), note that the quantity \(\|D_\mu f\|_{\rm Lip}\),
a priori,
might differ from the quantity 
\(\|D_\mu f\|\) given in \eqref{eq:intro_tot_var_tangential}.

\medskip

{\color{blue} (II)} \({\rm BV}(\R^d, \mu)={\rm BV}_{\rm Der}(\R^d, \mu)\):
In order to prove it, we first show in Section \ref{ssec:Der_VF}
that there exists an isometric isomorphism between
\({\rm D}_{\infty}({\rm div}_\mu)\)  and \({\rm Der}_b(\R^d, \mu)\). Due to this fact, 
we have that \(\|D_\mu f\|=\|D_\mu f\|_{\rm Der}\) for every \(f\in L^1_\mu(\R^d)\). 
This immediately implies \({\rm BV}(\R^d, \mu)\supseteq{\rm BV}_{\rm Der}(\R^d, \mu)\).
To get the opposite inclusion, we use the equivalent characterization of 
\(\|D_\mu f\|\) given in \eqref{eq:intro_tot_var_tangential} in order to 
construct the operator 
\(L_f\) as in point {\color{blue}(ii)} above, 
associated with \(f\in {\rm BV}(\R^d, \mu)\) (see Theorem \ref{thm:equiv_BV}).
Taking 
into account {\color{blue} (I)} and \eqref{eq:intro_DM_eqBV} we finally get 
\[
{\rm BV}(\R^d, \mu)={\rm BV}_{\rm Der}(\R^d, \mu)={\rm BV}_{\rm Lip}(\R^d, \mu)=
{\rm BV}_{C^\infty}(\R^d, \mu),
\]
and, moreover, that
\(|D_\mu f|_{\rm Der}=|D_\mu f|_{\rm Lip}=|D_\mu f|_{C^\infty}\) as measures, 
for every \(f\in {\rm BV}(\R^d, \mu)\).

\medskip

{\color{blue} (III)} \(W^{1,1}_{\rm Lip}(\R^d, \mu)\subseteq W^{1,1}(\R^d, \mu)\):
In the case of \(W^{1,1}\) spaces, there are not many instances where the various 
notions provided in \cite{DMPhD} do coincide. Also in this case we obtain (in 
Theorem \ref{thm:relation_W11})
only the above inclusion: to do so, we need to perform a careful study of 
the tangential gradient operator and its behaviour on 
compactly-supported Lipschitz functions. 
The entire Subsection \ref{ssec:W1,1} is devoted to this.
In particular, 
it allows us to give an equivalent characterization  of the space \(W^{1,1}(\R^d, \mu)\)
(see Theorem \ref{thm:equiv_char_W11}), which turns out to be more suitable 
for showing that \(W^{1,1}_{\rm Lip}(\R^d, \mu)\subseteq W^{1,1}(\R^d, \mu)\).
\medskip

\noindent{\bf Acknowledgements.}
The first named author acknowledges the support by the project PRIN 2017 ``Variational methods for stationary and evolution problems with singularities and interfaces'',
directed by Gianni Dal Maso (SISSA, Trieste, Italy).
The second named author acknowledges the support by the project 2017TEXA3H ``Gradient 
flows, Optimal Transport and Metric Measure Structures'', funded by 
the Italian Ministry of Research and University. 
We also wish to thank 
Simone Di Marino for the useful discussions on the topic.
\section{Preliminaries}
In this paper we are going to work in the setting of the weighted Euclidean space, namely,
in the space \(\R^d\) equipped with the Euclidean distance
\({\sf d}_{\rm Eucl}(x,y)\coloneqq|x-y|\) and an arbitrary 
non-negative Radon measure \(\mu\). The space \((\R^d,\sfd_{\rm Eucl},\mu)\) will be fixed to the end of the paper. Let us recall some basic notions that will be used throughout.
\smallskip

We denote by \(\LIP(\R^d)\) the space of all real-valued Lipschitz functions 
on \(\R^d\), whereas \(\LIP_c(\R^d)\) stands for the 
family of all elements of \(\LIP(\R^d)\) having  compact support. 
The Lipschitz constant of the restriction of a function
\(f\in \LIP(\R^d)\) to a set \(E\subseteq\R^d\) will be denoted by
\({\rm Lip}(f;E)\in[0,+\infty)\), while the global Lipschitz constant
of \(f\) will be denoted by \({\rm Lip}(f)\coloneqq{\rm Lip}(f;\R^d)\)
for brevity.

Given any \(p\in[1,\infty)\),
we denote by \(L^p_\mu(\R^d;\R^k)\) the space of \(p\)-integrable 
(with respect to \(\mu\)) \(\R^k\)-valued maps on \(\R^d\), while 
\(L^{\infty}_\mu(\R^d;\R^k)\) stands for the space of \(\mu\)-essentially
bounded \(\R^k\)-valued maps on \(\R^d\), in both cases considered up to
\(\mu\)-a.e.\ equality.
By \(L^0_\mu(\R^d)\) we shall denote the space of all \(\mu\)-measurable functions
on \(\R^d\), again considered up to \(\mu\)-a.e.\ equality.
It is well-known that the space \(L^p_\mu(\R^d;\R^k)\) is a
Banach space for any \(p\in[1,\infty]\), with respect to the norm
\[
\|v\|_{L^p_\mu(\R^d;\R^k)}\coloneqq\||v|\|_{L^p_\mu(\R^d)},
\quad \text{ for every } v\in L^p_\mu(\R^d;\R^k).
\]
\smallskip

A mollification argument will be often useful. Hence, we fix once and
for all a kernel of mollification \(\rho\) on \(\R^d\), \emph{i.e.},
a smooth, symmetric function \(\rho\in C^\infty_c(\R^d)\) such that
\(\rho\geq 0\), \({\rm supp}(\rho)\subseteq B_1(0)\), and
\(\int_{\mathbb R^d}\rho(x)\,\d\mathcal L^n(x)=1\). Given any
\(\varepsilon>0\), we define \(\rho_\varepsilon\in C^\infty_c(\R^d)\) as
\[
\rho_\varepsilon(x)\coloneqq\varepsilon^n\rho(x/\varepsilon),
\quad\text{ for every }x\in\R^d.
\]
Notice that \({\rm supp}(\rho_\varepsilon)\subseteq B_\varepsilon(0)\)
and \(\int_{\R^d}\rho_\varepsilon(x)\,\d\mathcal L^n(x)=1\).
Given a locally integrable, Borel function \(f\colon\R^d\to\R\),
we define its \(\varepsilon\)-mollification (with kernel \(\rho\)) as
the convolution between \(\rho_\varepsilon\) and \(f\), \emph{i.e.},
\[
(\rho_\varepsilon*f)(x)\coloneqq\int_{\R^d}\rho_\varepsilon(x-y)f(y)\,
\d\mathcal L^n(y)=\int_{\R^d}\rho_\varepsilon(y)f(x+y)\,\d\mathcal L^n(y),
\quad\text{ for every }x\in\R^d.
\]
In the following result we collect the main well-known properties of the
mollification:
\begin{lemma}[Approximation of compactly-supported Lipschitz functions]
\label{lemma:approx_lip_via_smooth}
Let \(f\in \LIP_{c}(\R^d)\) be given. Then for every \(\varepsilon>0\)
the \(\varepsilon\)-mollification \(f_\varepsilon\coloneqq
\rho_\varepsilon*f\in C^{\infty}_c(\R^d)\) satisfies
\begin{subequations}
\begin{align}
\label{eq:approx_lip_via_smooth_0}
{\rm supp}(f_\varepsilon)\subseteq B_\varepsilon({\rm supp}(f)),&\\
\label{eq:approx_lip_via_smooth_1}
\big|f_\varepsilon(x)-f(x)\big|\leq{\rm Lip}(f)\varepsilon,&
\quad\text{ for every }x\in \R^d,\\
\label{eq:approx_lip_via_smooth_2}
|\nabla f_\varepsilon(x)|\leq {\rm Lip}\big(f;B_{2\varepsilon}(x)\big),
&\quad \text{ for every }x\in \R^d.\end{align}
\end{subequations}
Moreover, it holds \(f=\lim_{\varepsilon\searrow 0}f_\varepsilon\)
strongly in \(L^p_\mu(\R^d)\) for any \(p\in[1,\infty)\) and
weakly\(^*\) in \(L^\infty_\mu(\R^d)\).
\end{lemma}
\begin{proof}
The fact that \(f_\varepsilon\in C^\infty(\R^d)\) is well-known.
Given any \(x\in\R^d\setminus B_\varepsilon({\rm supp}(f))\),
we have that \(x+y\notin{\rm supp}(f)\) for every \(y\in B_\varepsilon(0)\),
thus \(f_\varepsilon(x)=\int_{B_\varepsilon(0)}
\rho_\varepsilon(y)f(x+y)\,\d\mathcal L^n(y)=0\),
getting \eqref{eq:approx_lip_via_smooth_0} and in particular
that \(f_\varepsilon\in C^\infty_c(\R^d)\).
Now observe that for any \(x\in\R^d\) we may estimate
\[
\big|f_\varepsilon(x)-f(x)\big|\leq
\int_{\R^d}\big|f(x+y)-f(x)\big|\rho_\varepsilon(y)\,\d\mathcal L^n(y)
\leq {\rm Lip}(f) \int_{B_\varepsilon(0)}|y|\rho_\varepsilon(y)\,
\d\mathcal L^n(y)\leq {\rm Lip}(f)\varepsilon,
\]
which proves \eqref{eq:approx_lip_via_smooth_1}.
To verify  \eqref{eq:approx_lip_via_smooth_2}, take
\(y\in B_\varepsilon(x)\) with $y\neq x$. Then it holds that 
\[
\big|f_\varepsilon(y)-f_\varepsilon(x)\big|\leq
\int_{B_\varepsilon(0)}\big|f(y+v)-f(x+v)\big|\rho_\varepsilon(v)\,
\d\mathcal L^n(v)\leq {\rm Lip}\big(f;B_{2\varepsilon}(x)\big)|y-x|.
\]
By dividing the above inequality by \(|y-x|\) and passing to 
the limit as \(y\to x\), we get that 
\(|\nabla f_\varepsilon(x)|\leq {\rm Lip}\big(f;B_{2\varepsilon}(x)\big)\), 
proving \eqref{eq:approx_lip_via_smooth_2}. Finally,
for any \(\varepsilon\in(0,1)\) and \(x\in\R^d\) we have that
\[
|f_\varepsilon(x)|\leq
\int_{\R^d}|f(x+y)|\rho_\varepsilon(y)\,\d\mathcal L^n(y)
\leq\sup_{\R^d}|f|\int_{\R^d}\rho_\varepsilon\,\d\mathcal L^n
=\sup_{\R^d}|f|,
\]
which together with \eqref{eq:approx_lip_via_smooth_0} grant that
\(|f_\varepsilon|\leq\chi_K\|f\|_{L^\infty_\mu(\R^d)}\in
L^1_\mu(\R^d)\cap L^\infty_\mu(\R^d)\) for all \(\varepsilon\in(0,1)\),
where \(K\) stands for the closed \(1\)-neighbourhood of \({\rm supp}(f)\).
We know from \eqref{eq:approx_lip_via_smooth_1} that \(f_\varepsilon\)
pointwise converges to \(f\) as \(\varepsilon\searrow 0\),
so applying the dominated convergence theorem we get
\(f=\lim_{\varepsilon\searrow 0}f_\varepsilon\) strongly
in \(L^p_\mu(\R^d)\) for any \(p\in[1,\infty)\). For any
\(h\in L^1_\mu(\R^d)\) we have that \(|hf_\varepsilon|
\leq\chi_K|h|\|f\|_{L^\infty_\mu(\R^d)}\) holds \(\mu\)-a.e.\ for
every \(\varepsilon\in(0,1)\) and \((hf)(x)=
\lim_{\varepsilon\searrow 0}(hf_\varepsilon)(x)\) for
\(\mu\)-a.e.\ \(x\in\R^d\), thus by applying again the dominated
convergence theorem we conclude that \(\int_{\R^d}hf\,\d\mu
=\lim_{\varepsilon\searrow 0}\int_{\R^d}hf_\varepsilon\,\d\mu\).
Thanks to the arbitrariness of \(h\in L^1_\mu(\R^d)\), we conclude
that \(f=\lim_{\varepsilon\searrow 0}f_\varepsilon\)
weakly\(^*\) in \(L^\infty_\mu(\R^d)\).
\end{proof}

For the reader's usefulness, in the following statement we collect
some well-known fundamental results in functional analysis, which
will be used several times in the sequel.
\begin{proposition}\label{prop:FA_facts}
The following properties are verified:
\begin{itemize}
\item[\(\rm i)\)] Let \((v_n)_n\subseteq L^1_\mu(\R^d;\R^k)\) and
\(v\in L^1_\mu(\R^d;\R^k)\) be such that \(v_n\to v\) strongly in
\(L^1_\mu(\R^d;\R^k)\). Then some subsequence \((v_{n_i})_i\)
of \((v_n)_n\) is \emph{dominated}, i.e.\ there exists
\(g\in L^1_\mu(\R^d)\) such that \(|v_{n_i}|\leq g\) holds
\(\mu\)-a.e.\ for every \(i\in\N\). Moreover, we can further require that
\(v_{n_i}(x)\to v(x)\) as \(i\to\infty\) for \(\mu\)-a.e.\ \(x\in\R^d\).
\item[\(\rm ii)\)] Let \((v_n)_n\subseteq L^1_\mu(\R^d;\R^k)\) and
\(v\in L^1_\mu(\R^d;\R^k)\) be such that \(v_n\rightharpoonup v\)
weakly in \(L^1_\mu(\R^d;\R^k)\). Then for any \(n\in\N\) there exist
coefficients \((\alpha^n_i)_{i=n}^{N_n}\subseteq[0,1]\), for some
\(N_n\in\N\) with \(N_n\geq n\), such that
\(\sum_{i=n}^{N_n}\alpha^n_i=1\) and
\(\sum_{i=n}^{N_n}\alpha^n_i v_i\to v\) strongly in
\(L^1_\mu(\R^d;\R^k)\) as \(n\to\infty\).
\item[\(\rm iii)\)] Let \((v_n)_n\subseteq L^1_\mu(\R^d;\R^k)\) be a
dominated sequence. Then there exist \(v\in L^1_\mu(\R^d;\R^k)\)
and a subsequence \((v_{n_i})_i\) of \((v_n)_n\) such that \(v_{n_i}\rightharpoonup v\)
weakly in \(L^1_\mu(\R^d;\R^k)\) as \(i\to\infty\).
\end{itemize}
\end{proposition}
\begin{proof}
Observe that, writing \(v=(v^1,\ldots,v^k)\in L^1_\mu(\R^d;\R^k)\),
we may estimate
\[
\max_{i=1,\ldots,k}|v^i|\leq|v|=
\big(|v^1|^2+\ldots+|v^k|^2\big)^{1/2}\leq \sqrt{k}\max_{i=1,\ldots,k}|v^i|,
\quad\text{ in the }\mu\text{-a.e.\ sense.}
\]
In particular, a sequence \((v_n)_n\subseteq L^1_\mu(\R^d;\R^k)\)
is dominated if and only if \((v_n^i)_n\subseteq L^1_\mu(\R^d)\)
is dominated for every \(i=1,\ldots,k\). Moreover, it is easy
to check that \(v_n\) converges strongly (resp.\ weakly) in
\(L^1_\mu(\R^d;\R^k)\) to some vector field
\(v=(v^1,\ldots,v^k)\in L^1_\mu(\R^d;\R^k)\) if and only if
\(v_n^i\) converges strongly (resp.\ weakly) in \(L^1_\mu(\R^d)\)
to \(v^i\) for every \(i=1,\ldots,k\). Thanks to these observations,
we can prove the statement by arguing componentwise, i.e.\ it
suffices to deal with the case \(k=1\).\\
{\color{blue}i)} Fix any \((f_n)_n\subseteq L^1_\mu(\R^d)\) and
\(f\in L^1_\mu(\R^d)\) such that \(f_n\to f\) strongly in
\(L^1_\mu(\R^d)\). Then we can find a subsequence \((n_i)_i\)
satisfying \(\|f_{n_i}-f_{n_{i+1}}\|_{L^1_\mu(\R^d)}\leq 1/2^i\)
for every \(i\in\N\). Now let us define
\[
g(x)\coloneqq|f_{n_1}|(x)+\sum_{i=1}^\infty|f_{n_i}-f_{n_{i+1}}|(x),
\quad\text{ for }\mu\text{-a.e.\ }x\in\R^d.
\]
The \(\mu\)-a.e.\ defined functions \(g_j\coloneqq|f_{n_1}|+
\sum_{i=1}^j|f_{n_i}-f_{n_{i+1}}|\) satisfy \(g_j\nearrow g\)
in the \(\mu\)-a.e.\ sense and
\[
\int g_j\,\d\mu\leq\int|f_{n_1}|\,\d\mu+
\sum_{i=1}^j\int|f_{n_i}-f_{n_{i+1}}|\,\d\mu\leq
\|f_{n_1}\|_{L^1_\mu(\R^d)}+\sum_{i=1}^j\frac{1}{2^i}
\leq\|f_{n_1}\|_{L^1_\mu(\R^d)}+1.
\]
By using the monotone convergence theorem we get that
\(\int g\,\d\mu=\lim_j\int g_j\,\d\mu\leq\|f_{n_1}\|_{L^1_\mu(\R^d)}+1\),
thus in particular \(g\in L^1_\mu(\R^d)\). Notice also that for any
\(j\in\N\) with \(j\geq 2\) we have that
\[
|f_{n_j}|=\bigg|f_{n_1}+\sum_{i=2}^j(f_{n_i}-f_{n_{i-1}})\bigg|
\leq|f_{n_1}|+\sum_{i=2}^j|f_{n_i}-f_{n_{i-1}}|\leq g,\quad
\text{ in the }\mu\text{-a.e.\ sense.}
\]
All in all, we have proved that \((f_{n_i})_i\) is dominated by \(g\).
Finally, from the fact that \(g(x)<+\infty\) for \(\mu\)-a.e.\ \(x\in\R^d\)
we deduce that \(\big(f_{n_i}(x)\big)_i\subseteq\R\) is a Cauchy sequence
for \(\mu\)-a.e.\ \(x\in\R^d\). Given that \(f_{n_i}\to f\) strongly in
\(L^1_\mu(\R^d)\), we thus conclude that \(f_{n_i}(x)\to f(x)\) for
\(\mu\)-a.e.\ \(x\in\R^d\).\\
{\color{blue}ii)} Immediate consequence of Mazur Lemma, applied to the Banach space 
\(L^1_\mu(\R^d)\).\\
{\color{blue}iii)} It follows from Dunford--Pettis Theorem.
For a more direct proof, see \cite[Lemma 1.3.22]{GP_lectures}.
\end{proof}
\begin{remark}\label{rmk:Measures}
{\rm
With \(({\mathscr M}(\R^d),\|\cdot\|_{\sf TV})\) 
we denote the space of finite, signed Borel measures 
on \(\R^d\). Endowed with the total variation norm, denoted above by 
\(\|\cdot\|_{\sf TV}\), it results in a Banach space.
We recall that 
we can identify \(({\mathscr M}(\R^d),\|\cdot\|_{\sf TV})\) with the dual of the 
Banach space \(C_0(\R^d)\coloneqq {\rm cl}_{C_b(\R^d)}(C_c(\R^d))\). 
Here, \(C_b(\R^d)\) (resp.\ \(C_c(\R^d)\)) stands for the space
of bounded (resp.\ compactly-supported) continuous, real-valued functions on \(\R^d\).  
Recall also that \(C_b(\R^d)\) is a Banach space when endowed with the 
supremum norm \(\|f\|_{C_b(\R^d)}\coloneqq \sup_{x\in \R^d}|f(x)|\). 
\fr}
\end{remark}
\section{Different notions of BV space}
\subsection{BV space via vector fields}
The first attempt to the definition of BV functions in the setting of the
weighted Euclidean spaces has been done in \cite{BBF}. It is based on the 
notion of \(\mu\)-divergence of a vector field which we are going to
recall below.
\smallskip

A vector field \(v\in L^{\infty}_\mu(\R^d;\R^d)\)
is said to be a vector field with bounded \(\mu\)-divergence 
(in a distributional sense) if there
exists a function \({\rm div}_\mu(v)\in L^\infty_\mu(\R^d)\) such that
the following integration-by-parts formula holds:
\[
\int_{\R^d}\nabla f\cdot v\,{\rm d}\mu=-\int_{\R^d}f\,{\rm div}_\mu(v)\,{\rm d}\mu, \quad 
\text{ for every }f\in C^{\infty}_c(\R^d).
\]
Whenever it exists, \({\rm div}_\mu(v)\) is uniquely determined.
Let us define the space
\[
{\rm D}_\infty({\rm div}_\mu)\coloneqq
\big\{v\in L^\infty_\mu(\R^d;\R^d):\,v\text{ has bounded }\mu\text{-divergence}\big\}.
\]

The following definition of BV function has been proposed in \cite{BBF}:
\begin{definition}[BV space via vector fields]\label{def:BVEuclidean}
We say that a function 
\(f\in L^1_\mu(\R^d)\) 
belongs to the space \({\rm BV}(\R^d, \mu)\)
if the quantity 
\begin{equation}\label{eq:tot_var_eucl}
\|D_\mu f\|\coloneqq\sup\left\{\int_{\R^d}f\,{\rm div}_\mu(v)\,{\rm d}\mu:
\, v\in {\rm D}_{\infty}({\rm div}_\mu),\, |v|\leq 1\,\mu\text{-a.e.}\right\}
\end{equation}
is finite.
\end{definition}
The total variation defined in \eqref{eq:tot_var_eucl} can be localized
on open sets as follows: given any function \(f\in{\rm BV}(\R^d,\mu)\)
and any open set \(\Omega\subseteq\R^d\), we set
\[
|D_\mu f|(\Omega)\coloneqq\sup\left\{\int_\Omega f\,{\rm div}_\mu(v)
\,{\rm d}\mu:\, v\in {\rm D}_{\infty}({\rm div}_\mu),\, {\rm supp}(v)
\Subset\Omega,\, |v|\leq 1\,\mu\text{-a.e.}\right\}.
\]
Notice that \(|D_\mu f|(\R^d)=\|D_\mu f\|\) by definition.
The function \(|D_\mu f|\) can be extended to all Borel sets
\(B\subseteq\R^d\) via Carath\'{e}odory construction, as follows:
\[
|D_\mu f|(B)\coloneqq\inf\big\{|D_\mu f|(\Omega):\,\Omega
\subseteq\R^d\text{ open},\, B\subseteq\Omega\big\}.
\]
It turns out that \(|D_\mu f|\) is a finite Borel measure on
\(\R^d\). However, we do not verify it right now; we will obtain
it as a consequence of Theorem \ref{thm:equiv_BV}.
\subsection{BV space via relaxation}\label{ssec:BV_relaxed}
The relaxation-type approach to the definition of BV space
has been firstly introduced in \cite{MIRANDA03} in the setting of
metric measure spaces. We shall present here a slight variant of it, which has been proposed in \cite{DMPhD}. 
\smallskip

Given any open set \(\Omega\subseteq \R^d\), we denote by
\(\LIP_{loc}(\Omega)\) the family of all locally Lipschitz functions on \(\Omega\),
i.e.\ those functions \(f\colon \Omega\to \R\) satisfying the following: 
for every \(x\in \Omega\) there exists 
\(r_x>0\) such that \(f|_{B_{r_x}(x)}\) is Lipschitz. 
Given any \(f\in\LIP_{loc}(\Omega)\), we shall denote 
by \(\lip_a(f)\colon \Omega\to[0,+\infty)\)
its \emph{asymptotic Lipschitz constant},
which is defined as
\begin{equation}\label{eq:lip_a}
\lip_a(f)(x)\coloneqq\lim_{r\to 0}{\rm Lip}\big(f;B_r(x)\big)
=\lims_{\substack{y\neq z\\ y,z\to x}}\frac{|f(y)-f(z)|}{{\sf d}(y,z)},
\end{equation}
if \(x\in \Omega\) is an accumulation point and \({\rm lip}_a(f)(x)\coloneqq 0\) 
otherwise.
\smallskip

Taking into account \cite[Theorem 4.5.3]{DMPhD}, we have the following definition of
BV space:
\begin{definition}[BV space via relaxation]\label{def:BVmms}
We say that a function 
\(f\in L^1_\mu(\R^d)\) 
belongs to the space \({\rm BV}_{{\rm Lip}}(\R^d, \mu)\),
if one of the following equivalent 
conditions is satisfied:
\begin{itemize}
\item [1)] 
There exists
a sequence \((f_n)_n\subseteq \LIP_{loc}(\R^d)\cap L^1_\mu(\R^d)\) such that 
\[f_n\to f\;\text{ in } L^1_{\mu}(\R^d)\quad
\text{ and } \quad\sup_n\int_{\R^d}\lip_a(f_n)\,{\rm d}\mu<+\infty.\]
\item [2)] 
There exists
a sequence \((f_n)_n\subseteq \LIP_c(\R^d)\) such that 
\[f_n\to f\;\text{ in } L^1_\mu(\R^d)\quad
\text{ and } \quad\sup_n\int_{\R^d}\lip_a(f_n)\,{\rm d}\mu<+\infty.\]
\end{itemize}
\end{definition}

Given any \(f\in {\rm BV}_{\rm Lip}(\R^d,\mu)\), the
\emph{total variation measure} \(|D_\mu f|_{\rm Lip}\)
associated with \(f\) is defined in \cite{DMPhD} as
\[
|D_\mu f|_{\rm Lip}(B)\coloneqq\inf\big\{|D_\mu f|_{\rm Lip}(\Omega):
\,\Omega\subseteq\R^d\text{ open},\,B\subseteq\Omega\big\},\quad
\text{ for every }B\subseteq\R^d\text{ Borel,}
\]
where for any open set \(\Omega\subseteq \R^d\) we set
\begin{equation}\label{eq:tot_variation_local_mms}
|D_\mu f|_{\rm Lip}(\Omega)\coloneqq 
\inf\Big\{\limi_{n\to\infty}\int_\Omega\lip_a(f_n)\,{\rm d}\mu:
\;(f_n)_n\subseteq \LIP_{loc}(\Omega)\cap L^1_\mu(\Omega),\; 
f_n\to f\text{ in }L^1_\mu(\Omega)\Big\}.
\end{equation}
The \emph{total variation} of \(f\), \emph{i.e.}, the total variation measure 
evaluated at the entire space,
can be recovered by using only compactly-supported Lipschitz functions 
(cf.\ \cite[Theorem 4.5.3]{DMPhD}), 
namely:
\begin{equation}\label{eq:tot_variation_mass_mms}
|D_\mu f|_{\rm Lip}(\R^d)= \inf\Big\{\limi_{n\to\infty}
\int_{\R^d}\lip_a(f_n)\,{\rm d}\mu:\;(f_n)_n\subseteq \LIP_c(\R^d),
\; f_n\to f\text{ in }L^1_\mu(\R^d)\Big\}.
\end{equation}
As observed in the example preceding \cite[Proposition 4.4.1]{DMPhD}, 
the formulation in
\eqref{eq:tot_variation_mass_mms}, using compactly-supported
Lipschitz functions instead of locally Lipschitz ones,
cannot be used (in general) to compute the quantity
\(|D_\mu f|_{\rm Lip}(\Omega)\) for any open set
\(\Omega\subseteq\R^d\).
\begin{remark}{\rm
While Definition \ref{def:BVEuclidean} is tailored to the
weighted Euclidean space setting, the concept in Definition
\ref{def:BVmms} (as well as the one in Definition
\ref{def:BV_derivation}) actually makes sense on any metric measure
space \((\X,\sfd,\mathfrak m)\), \emph{i.e.}, \((\X,\sfd)\)
is a complete separable metric space and \(\mathfrak m\)
is a boundedly-finite Borel measure on \(\X\);
we refer to \cite{DMPhD} for the details. This remark
will play a role in Section \ref{ssec:total_var_meas}.
\fr}\end{remark}
\subsection{BV space via derivations}\label{ssec:BV_der}
In this subsection we report the definition of BV space via derivations 
proposed in \cite{DMPhD}. 
We start by recalling the definition of derivation introduced in \cite{DMPhD} 
and point out some of its basic properties.
\begin{definition}\label{def:derivation}
By a  \emph{derivation} on \((\R^d,{\sf d}_{\rm Eucl},\mu)\)
we mean any linear map \(\mathbf{b}\colon \LIP_c(\R^d)\to L^0_\mu(\R^d)\) 
satisfying the following two properties:
\begin{itemize}
\item [1)] {\color{blue}\sc Leibniz rule:} For every \(f,g\in \LIP_c(\R^d)\),
it holds that
\[
\mathbf{b}(fg)=\mathbf{b}(f)g+f\mathbf{b}(g).
\]
\item [2)] {\color{blue}\sc Weak locality:} There exists a non-negative function 
\(G\in L^0_\mu(\R^d)\)
such that 
\[|\mathbf{b}(f)|\leq G\,\lip_a(f)
\quad \text{ holds }\mu\text{-a.e.,\ }\text{ for every }f\in \LIP_c(\R^d).\]
\end{itemize}
The least function \(G\) as above will be denoted by \(|\mathbf b|\).
\end{definition}

We recall from \cite{GDMSP} that for a given derivation 
\(\mathbf{b}\), one has the following formula for 
\(|\mathbf{b}|\):
\begin{equation}\label{eq:formula_for_|b|}
|\mathbf{b}|=\text{\rm ess sup}\,\big\{\mathbf{b}(f):\,f\in \LIP_c(\R^d),\,
 {\rm Lip}(f)\leq 1\big\}.
\end{equation}
By the \emph{support} of a derivation \(\mathbf b\) we mean the support of 
the associated function \(|\mathbf{b}|\), and we denote it by \({\rm supp}(\mathbf{b})\).

We denote by 
\({\rm Der}(\R^d,\mu)\) the space of all derivations on 
\((\R^d,\sf d_{\rm Eucl},\mu)\), while
\({\rm Der}_{\infty}(\R^d,\mu)\) stands for the space of all 
\emph{bounded derivations}, \emph{i.e.},
\(
{\rm Der}_{\infty}(\R^d,\mu)\coloneqq
\big\{\mathbf{b}\in {\rm Der}(\R^d,\mu):\, |\mathbf{b}|\in L^{\infty}_\mu(\R^d)\big\}.
\)
We have that \(\big({\rm Der}_{\infty}(\R^d,\mu),\|\cdot\|_b\big)\) is a Banach space,
where we set
\[
\|\mathbf b\|_b\coloneqq \||\mathbf b|\|_{L^{\infty}_\mu(\R^d)},\quad\text{ for every }
\mathbf b\in {\rm Der}_{\infty}(\R^d,\mu).
\]
In what follows we will be concentrated on those elements 
\(\mathbf{b}\in{\rm Der}_{\infty}(\R^d,\mu)\) admitting
\emph{bounded divergence}, \emph{i.e.}, for which there
is a (uniquely determined) function
\({\rm div}(\mathbf{b})\in L^\infty_\mu(\R^d)\) such that 
\[
\int_{\R^d}\mathbf{b}(f)\,{\rm d}\mu=
-\int_{\R^d}f\,{\rm div}(\mathbf{b})\,{\rm d}\mu,
\quad \text{ for every }f\in \LIP_c(\R^d).
\]
The space of all bounded derivations with bounded divergence will be denoted by 
\({\rm Der}_{b}(\R^d,\mu)\), which is a Banach space when endowed with the norm
\[
\|\mathbf{b}\|_{b,b}\coloneqq
\|\mathbf{b}\|_b+
\|{\rm div}(\mathbf{b})\|_{L^{\infty}_\mu(\R^d)}.
\]
Let us recall the Leibniz rule for the divergence: given any
\(h\in\LIP_c(\R^d)\) and \(\mathbf b\in{\rm Der}_b(\R^d,\mu)\),
it holds that \(h\mathbf b\in{\rm Der}_b(\R^d,\mu)\) and
\begin{equation}\label{eq:Leibniz_div}
{\rm div}(h\mathbf b)=\mathbf b(h)+h\,{\rm div}(\mathbf b).
\end{equation}
\begin{definition}\label{def:BV_derivation}
We say that a function 
\(f\in L^1_\mu(\R^d)\) 
belongs to the space \({\rm BV}_{\rm Der}(\R^d,\mu)\) 
if there exists
a \(\LIP_c(\R^d)\)-linear and \(\|\cdot\|_b\)-continuous map 
\(D f\colon {\rm Der}_b(\R^d,\mu)\to \mathscr M(\R^d)\) satisfying
\[
\int_{\R^d}{\rm d}D f(\mathbf{b})=-\int_{\R^d}f\,{\rm div}(\mathbf{b})\,{\rm d}\mu, 
\quad \text{ for every }\mathbf{b}\in {\rm Der}_b(\R^d,\mu).
\]
In this case, the map \(D f\) is uniquely determined.
\end{definition}

Given any \(f\in {\rm BV}_{\rm Der}(\R^d,\mu)\), the \emph{total variation measure} 
associated with \(f\) is defined as the unique finite Borel measure
\(|D_\mu f|_{\rm Der}\) on \(\R^d\) that for each open set
\(\Omega\subseteq \R^d\) satisfies
\[
|D_\mu f|_{\rm Der}(\Omega)=
\sup\left\{\int_\Omega f\,{\rm div}(\mathbf{b})\,{\rm d}\mu:\,
\mathbf{b}\in {\rm Der}_b(\R^d,\mu),\,{\rm supp}(\mathbf{b})\Subset \Omega,\, 
|\mathbf{b}|\leq 1\; 
\mu\text{-a.e.}\right\}.
\]
\section{Different notions of \texorpdfstring{\(W^{1,1}\)}{W11} space}
\subsection{\texorpdfstring{\(W^{1,1}\)}{W11}
space via vector fields}\label{ssec:W1,1}
An approach based on a notion of a 
`space tangent to a measure' (and in turn on the properties of
vector fields with divergence) has been used in the pioneering work \cite{BBS} 
to propose a concept of  Sobolev space \(W^{1,p}\) in the case \(p\in (1,\infty)\).
As observed in \cite{BBF}, the very same technique may be applied in the case \(p=1\).
Below, we recall the definition of \(W^{1,1}\) space from \cite{BBF} and study more 
in details the properties of the `tangential gradient operator' which plays a crucial
role in its definition. 
\medskip

First of all, let us recall the definition of a (measurable) bundle
in \(\R^d\), following quite closely the presentation in
\cite{LPR}. Let \(V\) be a map assigning to any point \(x\in\R^d\) a
vector subspace \(V(x)\) of \(\R^d\). Then we say that \(V\) is a
(measurable) bundle in \(\R^d\) provided
\(\R^d\ni x\mapsto{\sf d}_{\rm Eucl}(y,V(x))\in\R\) is Borel measurable
for any \(y\in\R^d\). A partial order (depending on \(\mu\)) on the
family of all bundles in \(\R^d\) is given as follows: if \(V\) and
\(W\) are bundles in \(\R^d\), we declare that \(V\preceq W\) provided
\(V(x)\subseteq W(x)\) for \(\mu\)-a.e.\ \(x\in\R^d\). Given an exponent
\(p\in[1,\infty]\) and a bundle \(V\) in \(\R^d\), we denote by
\(\Gamma^p_\mu(V)\) the space of all \(L^p_\mu(\R^d)\)-sections of \(V\),
namely,
\[
\Gamma^p_\mu(V)\coloneqq\big\{v\in L^p_\mu(\R^d;\R^d)\,:\,
v(x)\in V(x)\text{ for }\mu\text{-a.e.\ }x\in\R^d\big\}.
\]
Observe that \(\Gamma^p_\mu(V)\) is a closed vector subspace of
\(L^p_\mu(\R^d;\R^d)\) which is closed under multiplication
by \(L^\infty_\mu(\mu)\)-functions.
As proven in \cite[Proposition 2.22]{LPR}, it holds that
\begin{equation}\label{eq:V_leq_W}
V\preceq W\quad\Longleftrightarrow\quad
\Gamma^2_\mu(V)\subseteq\Gamma^2_\mu(W),
\end{equation}
whenever \(V\) and \(W\) are bundles in \(\R^d\).
\begin{lemma}
There exists a \(\preceq\)-minimal bundle \(T_\mu\) in \(\R^d\),
uniquely determined up to \(\mu\)-a.e.\ equality, such that
\begin{equation}\label{eq:def_T_mu}
\text{given any }v\in{\rm D}_\infty({\rm div}_\mu),
\text{ it holds that }v(x)\in T_\mu(x)\text{ for }
\mu\text{-a.e.\ }x\in\R^d.
\end{equation}
\end{lemma}
\begin{proof}
We subdivide the proof into three steps:\\
{\color{blue}\textsc{Step 1.}} First of all, let us define
\begin{equation}\label{eq:V_def}
\mathcal V\coloneqq
\big\{v\in {\rm D}_{\infty}({\rm div}_\mu):\, |v|\in L^2_\mu(\R^d)\big\} 
\quad \text{ and }\quad
\mathcal M\coloneqq {\rm cl}_{L^2_\mu(\R^d;\R^d)}(\mathcal V).
\end{equation}
We claim that
\begin{equation}\label{eq:T_mu_aux1}
v\in{\rm D}_\infty({\rm div}_\mu)\,\text{ and }\,f\in C^\infty_c(\R^d)
\quad\Longrightarrow\quad fv\in\mathcal V\,\text{ and }
\,{\rm div}_\mu(fv)=f\,{\rm div}_\mu(v)+\nabla f\cdot v.
\end{equation}
To prove it, notice that for every
\(g\in C^\infty_c(\R^d)\) we have that \(fg\in C^\infty_c(\R^d)\)
and \(\nabla(fg)=f\nabla g+g\nabla f\), thus accordingly
\[\begin{split}
\int_{\R^d}\nabla g\cdot(fv)\,\d\mu&=
\int_{\R^d}(f\nabla g)\cdot v\,\d\mu=\int_{\R^d}
\nabla(fg)\cdot v\,\d\mu-\int_{\R^d}(g\nabla f)\cdot v\,\d\mu\\
&=-\int_{\R^d}g\big(f\,{\rm div}_\mu(v)+\nabla f\cdot v\big)\,\d\mu.
\end{split}\]
Since \(fv\in L^\infty_\mu(\R^d;\R^d)\cap L^2_\mu(\R^d;\R^d)\) and
\(f\,{\rm div}_\mu(v)+\nabla f\cdot v\in L^\infty_\mu(\R^d)\),
we have obtained \eqref{eq:T_mu_aux1}.\\
{\color{blue}\textsc{Step 2.}} Next we claim that
\begin{equation}\label{eq:T_mu_aux2}
v\in\mathcal M\,\text{ and }\,f\in L^\infty_\mu(\R^d)
\quad\Longrightarrow\quad fv\in\mathcal M.
\end{equation}
To prove it, fix \((v_n)_n\subseteq\mathcal V\) such that
\(v_n\to v\) in \(L^2_\mu(\R^d;\R^d)\). Moreover, we can find
\((f_n)_n\subseteq C^\infty_c(\R^d)\) such that \(|f_n|\leq\|f\|_{L^\infty_\mu(\R^d)}\)
for every \(n\in\N\) and \(f_n\to f\) in the \(\mu\)-a.e.\ sense.
Indeed, chosen a finite Borel measure \(\tilde\mu\) on \(\R^d\)
having the same null sets as \(\mu\), there exists a sequence
\((g_n)_n\subseteq\LIP_c(\R^d)\) satisfying \(g_n\to f\)
in \(L^2_{\tilde\mu}(\R^d)\). Up to replacing \(g_n\) with
\((g_n\wedge\|f\|_{L^\infty_\mu(\R^d)})\vee(-\|f\|_{L^\infty_\mu(\R^d)})\),
we can assume that \(|g_n|\leq\|f\|_{L^\infty_\mu(\R^d)}\) holds
\(\mu\)-a.e.. Up to passing to a not relabeled subsequence, we can
further assume that \(g_n\to f\) in the \(\tilde\mu\)-a.e.\ sense
(thus, in the \(\mu\)-a.e.\ sense). Now choose \((\varepsilon_n)_n
\subseteq(0,1)\) such that each function
\(f_n\coloneqq\rho_{\varepsilon_n}*g_n\in C^\infty_c(\R^d)\)
satisfies \(\|f_n-g_n\|_{C_b(\R^d)}\leq 1/n\).
In particular, it holds that \(|f_n|\leq\|f\|_{L^\infty_\mu(\R^d)}\)
for every \(n\in\N\) and \(f_n\to f\) in the \(\mu\)-a.e.\ sense,
as desired. Therefore, we can estimate
\[\begin{split}
\|f_n v_n-fv\|_{L^2_\mu(\R^d;\R^d)}^2&\leq
2\int_{\R^d}|f_n-f|^2|v|^2\,\d\mu+2\int_{\R^d}|f_n|^2|v_n-v|^2\,\d\mu\\
&\leq 2\int_{\R^d}|f_n-f|^2|v|^2\,\d\mu+2\,\|f\|_{L^\infty_\mu(\R^d)}^2
\|v_n-v\|_{L^2_\mu(\R^d;\R^d)}^2.
\end{split}\]
Since \(|f_n-f|^2|v|^2\leq 2\,\|f\|_{L^\infty_\mu(\R^d)}^2|v|^2
\in L^1_\mu(\R^d)\) for every \(n\in\N\), by dominated convergence
theorem we deduce that \(f_n v_n\to fv\) in \(L^2_\mu(\R^d;\R^d)\).
As \((f_n v_n)_n\subseteq\mathcal V\) by \eqref{eq:T_mu_aux1}, we
conclude that \(fv\in\mathcal M\).\\
{\color{blue}\textsc{Step 3.}} We are now in a position to apply
\cite[Proposition 2.22]{LPR}:
\eqref{eq:T_mu_aux2} grants that \(\mathcal M\) is a
\(L^2_\mu(\R^d)\)-normed \(L^\infty_\mu(\R^d)\)-submodule of
\(L^2_\mu(\R^d;\R^d)\) in the sense of \cite[Definition 1.2.10]{gigli2018nonsmooth},
thus there exists a unique bundle \(T_\mu\) in \(\R^d\) such
that \(\Gamma^2_\mu(T_\mu)=\mathcal M\). To show that
\(T_\mu\) satisfies \eqref{eq:def_T_mu}, let us fix
\(v\in {\rm D}_\infty({\rm div}_\mu)\) and a sequence 
\((\eta_n)_n\subseteq C^\infty_c(\R^d)\) such that 
\(0\leq\eta_n\leq 1\) and \(\eta_n=1\) on \(B_n(0)\)
for every \(n\in \N\). Fix \(n\in \N\) and notice that
\(\eta_n v\in \mathcal V\subseteq\Gamma^2_\mu(T_\mu)\)
by \eqref{eq:T_mu_aux1}. Hence, \(v(x)=(\eta_n v)(x)\in T_\mu(x)\)
for \(\mu\)-a.e.\ \(x\in B_n(0)\). Thanks to the arbitrariness of
\(n\in \N\), we obtain that \(T_\mu\) satisfies \eqref{eq:def_T_mu}.

Finally, we are left to prove the minimality of \(T_\mu\), which
also forces uniqueness. Fix an arbitrary bundle \(S\) in \(\R^d\)
satisfying the property in \eqref{eq:def_T_mu} with \(S(x)\) in
place of \(T_\mu(x)\). We aim to show that \(T_\mu(x)\subseteq S(x)\)
for \(\mu\)-a.e.\ \(x\in \R^d\). Taking into account \eqref{eq:V_leq_W},
we can equivalently show that \(\Gamma^2_\mu(T_\mu)\subseteq \Gamma^2_\mu(S)\). 
Pick any \(v\in \Gamma^2_\mu(T_\mu)\) and a sequence
\((v_n)_n\subseteq \mathcal V\) such that \(v_n\to v\) in
\(L^2_\mu(\R^d;\R^d)\). Up to a not relabeled subsequence, we have that
\(v_n(x)\to v(x)\) for \(\mu\)-a.e.\ \(x\in \R^d\). Given that for
\(\mu\)-a.e.\ \(x\in \R^d\) it holds that \(v_n(x)\in S(x)\) for every
\(n\in \N\), we conclude that \(v(x)\in S(x)\) for \(\mu\)-a.e.\ \(x\in \R^d\).
\end{proof} 
\begin{remark}\label{rmk:family_C}{\rm
Let us point out that there exists a countable family 
\(\mathcal C\subseteq {\rm D}_\infty({\rm div}_\mu)\), such that
\[
(w(x))_{w\in \mathcal C}\, \text{ is dense in } T_\mu(x)\,
\text{ for }\mu\text{-a.e.\ }x\in \R^d.
\] 
To verify this, observe that \(\mathcal V\subseteq {\rm D}_\infty({\rm div}_\mu)\)
defined in \eqref{eq:V_def} is a linear subspace of \(L^2_\mu(\R^d;\R^d)\). Moreover, 
it is closed under the multiplication by \(C^\infty_c(\R^d)\)-functions, 
due to \eqref{eq:T_mu_aux1}. Take now any countable \(L^2_\mu(\R^d;\R^d)\)-dense 
subset \(\mathcal C\subseteq \mathcal V\subseteq {\rm D}_\infty({\rm div}_\mu)\) and define 
\[
V(x)\coloneqq {\rm cl}\big(\{w(x):\, w\in \mathcal C\}\big)\, 
\text{ for }\mu\text{-a.e.\ }x\in \R^d.
\]
By applying \cite[Lemma 2.24]{LPR} (see also \cite[Lemma A.1]{BF_Second_order}), 
we have that 
\(\Gamma^2_\mu(V)={\rm cl}_{L^2_\mu(\R^d;\R^d)}(\mathcal V)=\mathcal M\). 
On the other hand, we have from the construction of the bundle \(T_\mu\) that 
\(\mathcal M=\Gamma^2_\mu(T_\mu)\). Thus, \(\Gamma^2_\mu(T_\mu)=\Gamma^2_\mu(V)\) which 
further implies (recalling \cite[Proposition 2.22]{LPR}) that \(T_\mu=V\). 
Therefore, 
\(\{w(x):\, w\in \mathcal C\}\) is dense in \(T_\mu(x)\) for \(\mu\)-a.e.\ 
\(x\in \R^d\), as claimed.
\fr}\end{remark}

Once we have the tangent fibers at our disposal, we are in a position to 
define the \emph{tangential gradient}: namely, we define 
\(\nabla_{\mu}\colon C^{\infty}_c(\R^d)\to \Gamma^1_\mu(T_\mu)\) as 
\begin{equation}\label{eq:tangential_grad_smooth}
\nabla_{\mu}f(x)\coloneqq {\rm pr}_{T_\mu(x)}\big(\nabla f(x)\big),\quad 
\text{ for every }f\in C^{\infty}_c(\R^d)\text{ and }\mu\text{-a.e. } x\in\R^d,
\end{equation}
where \({\rm pr}_{V}\colon \R^d\to V\) stands for the orthogonal projection onto the 
vector subspace \(V\) of \(\R^d\).
\smallskip

In \cite{BBF} the space of \(W^{1,1}\) functions 
has been defined as follows:

\begin{definition}[\(W^{1,1}\) space via vector fields]
The Sobolev space \(W^{1,1}(\R^d,\mu)\) is defined as the completion of 
\(C^{\infty}_c(\R^d)\) under the norm
\[
\|f\|_{W^{1,1}_\mu}\coloneqq \|f\|_{L^1_\mu(\R^d)}+
\|\nabla_{\mu} f\|_{\Gamma^1_\mu(T_\mu)}.
\]
\end{definition}

Observe that, a priori, \(\big(W^{1,1}(\R^d,\mu),\|\cdot\|_{W^{1,1}_\mu}\big)\)
is an abstract Banach space. Let us now show that \(W^{1,1}(\R^d,\mu)\)
is actually a space of functions, \emph{i.e.}, that it can be
identified with a linear subspace of \(L^1_\mu(\R^d)\). The inclusion
map \(\big(C^\infty_c(\R^d),\|\cdot\|_{W^{1,1}_\mu}\big)\hookrightarrow
\big(L^1_\mu(\R^d),\|\cdot\|_{L^1_\mu(\R^d)}\big)\) is a linear
contraction. Denoting by \(\iota\colon C^\infty_c(\R^d)
\hookrightarrow W^{1,1}(\R^d,\mu)\) the canonical isometric embedding
of \(\big(C^\infty_c(\R^d),\|\cdot\|_{W^{1,1}_\mu}\big)\) into its
completion \(W^{1,1}(\R^d,\mu)\), there exists a unique linear
contraction \(\phi\colon W^{1,1}(\R^d,\mu)\to L^1_\mu(\R^d)\)
such that the following diagram is commutative:
\begin{equation}\label{eq:def_phi}\begin{tikzcd}
C^\infty_c(\R^d) \arrow[d,swap,"\iota",hook] \arrow[r,hook] &
L^1_\mu(\R^d) \\ W^{1,1}(\R^d,\mu) \arrow[ru,swap,"\phi"] &
\end{tikzcd}\end{equation}
Our aim is to prove that \(\phi\) is injective.
To achieve this goal, we need the following key result.
\begin{lemma}[Closability of the tangential gradient]\label{lem:grad_mu_closed} 
Let 
\((f_n)_{n\in \N}\subseteq C^\infty_c(\R^d)\) and \(v\in\Gamma^1_\mu(T_\mu)\) 
be such that 
\[
f_n\rightharpoonup 0\, \text{ in }L^1_\mu(\R^d)\quad \text{ and }
\quad \nabla_\mu f_n\rightharpoonup v\, \text{ in }\Gamma^1_\mu(T_\mu). 
\]
Then \(v(x)=0\) for \(\mu\)-a.e.\ \(x\in \R^d\).
\end{lemma}
\begin{proof}
Fix \(g\in C^{\infty}_c(\R^d)\) and \(w\in \mathcal C\), 
where \(\mathcal C\subseteq D_{\infty}({\rm div}_\mu)\) is a countable family 
as in Remark \ref{rmk:family_C}. 
Then, taking into account the property \eqref{eq:T_mu_aux1}, we have that
\[
\int_{\R^d}g\,w\cdot v\,\d\mu=
\lim_{n\to\infty}\int_{\R^d} g\, w\cdot \nabla_\mu f_n\,\d\mu=
-\lim_{n\to\infty}\int_{\R^d}f_n\,{\rm div}_\mu(g\,w)\,\d\mu=0.
\]
Now, let \(f\in \LIP_c(\R^d)\) and let \((g_n)_n\subseteq C^{\infty}_c(\R^d)\) 
be such that
\(g_n\to f\) pointwise \(\mu\)-a.e.\ and \(|g_n|\leq \|f\|_{L^\infty_\mu(\R^d)}\)
for every \(n\in \N\) 
(the existence of such a sequence follows from a standard mollification argument, 
cf.\ Lemma \ref{lemma:approx_lip_via_smooth}).
Then, by applying the dominated convergence theorem and using the above equality,
we get that
\[
\int_{\R^d}f\,w\cdot v\,\d\mu=
\lim_{n\to\infty}\int_{\R^d}g_n\, w\cdot v\,\d\mu=0.
\]
Since the function \(f\in \LIP_c(\R^d)\) was arbitrary, we deduce that \(w\cdot v=0\)
holds \(\mu\)-a.e.\ in \(\R^d\).  
By Remark \ref{rmk:family_C}, we know that the elements of the family \(\mathcal C\) 
are fiberwise dense in \(\mu\)-a.e.\ fiber of \(T_\mu\). 
Thus, by the arbitrariness of \(w\in \mathcal C\), 
we finally conclude that \(v(x)=0\) for \(\mu\)-a.e.\ \(x\in \R^d\).
\end{proof}
\begin{corollary}\label{cor:W11_space_fcts}
The map \(\phi\) as in \eqref{eq:def_phi} is injective.
In particular, \(W^{1,1}(\R^d,\mu)\) can be identified with
a linear subspace of \(L^1_\mu(\R^d)\).
\end{corollary}
\begin{proof}
To prove the claim amounts to showing that if
\({\sf f}\in W^{1,1}(\R^d,\mu)\) and \(\phi({\sf f})=0\),
then \({\sf f}=0\); we are using the different font
\(\sf f\) to underline that, a priori, the elements of
\(W^{1,1}(\R^d,\mu)\) are not functions. Choose a sequence
\((f_n)_{n\in\N}\subseteq C^\infty_c(\R^d)\) such that
\(\|\iota(f_n)-{\sf f}\|_{W^{1,1}_\mu}\to 0\). In particular,
the sequences \((f_n)_{n\in\N}\subseteq L^1_\mu(\R^d)\) and
\((\nabla_\mu f_n)_{n\in\N}\subseteq\Gamma^1_\mu(T_\mu)\) are
Cauchy, thus there exist elements \(f\in L^1_\mu(\R^d)\) and
\(v\in\Gamma^1_\mu(T_\mu)\) such that \(f_n\to f\) and
\(\nabla_\mu f_n\to v\). Being \(\phi\) continuous, we have that
\(f_n=\phi(\iota(f_n))\to\phi({\sf f})=0\) in \(L^1_\mu(\R^d)\),
whence it follows that \(f=0\). Hence, an application of Lemma
\ref{lem:grad_mu_closed} yields the identity \(v=0\). All in all,
we proved that \(\|f_n\|_{W^{1,1}_\mu}\to 0\), so that \({\sf f}=0\).
\end{proof}

In light of Corollary \ref{cor:W11_space_fcts}, hereafter we will
tacitly regard \(W^{1,1}(\R^d,\mu)\) as a subspace of
\(L^1_\mu(\R^d)\). Next we show that the tangential gradient
\(\nabla_\mu\) can be extended to the whole of \(W^{1,1}(\R^d,\mu)\):
\begin{proposition}\label{prop:ext_tg_grad}
There exists a unique linear extension
\(\nabla_\mu\colon W^{1,1}(\R^d,\mu)\to\Gamma_\mu^1(T_\mu)\)
of the tangential gradient \(\nabla_\mu\colon C^\infty_c(\R^d)
\to\Gamma^1_\mu(T_\mu)\) having the following property:
if \((f_n)_{n\in\N}\subseteq W^{1,1}(\R^d,\mu)\) satisfies
\(f_n\rightharpoonup f\) in \(L^1_\mu(\R^d)\) and
\(\nabla_\mu f_n\rightharpoonup v\) in \(\Gamma^1_\mu(T_\mu)\)
for some \(f\in L^1_\mu(\R^d)\) and \(v\in\Gamma^1_\mu(T_\mu)\),
then \(f\in W^{1,1}(\R^d,\mu)\) and \(\nabla_\mu f=v\).
Moreover, it holds that
\begin{equation}\label{eq:formula_W11_norm}
\|f\|_{W^{1,1}_\mu}=\|f\|_{L^1_\mu(\R^d)}+
\|\nabla_\mu f\|_{\Gamma^1_\mu(\R^d)},\quad
\text{ for every }f\in W^{1,1}(\R^d,\mu).
\end{equation}
\end{proposition}
\begin{proof}
Let \(f\in W^{1,1}(\R^d,\mu)\) be given. Pick any sequence
\((f_n)_{n\in\N}\subseteq C^\infty_c(\R^d)\) such that
\(f_n\to f\) strongly in \(W^{1,1}(\R^d,\mu)\). In particular,
\((\nabla_\mu f_n)_{n\in\N}\) is a Cauchy sequence in \(\Gamma_\mu^1(T_\mu)\).
Then we define \(\tilde \nabla_\mu f\in\Gamma^1_\mu(T_\mu)\) as the
limit of \(\nabla_\mu f_n\) as \(n\to\infty\). Notice that this
definition is well-posed, \emph{i.e.}, \(\tilde\nabla_\mu f\) does not
depend on the specific choice of \((f_n)_n\): indeed, given another
sequence \((g_n)_{n\in\N}\subseteq C^\infty_c(\R^d)\) such that
\(g_n\to f\) in \(W^{1,1}(\R^d,\mu)\), we have that \(f_n-g_n\to 0\)
in \(L^1_\mu(\R^d)\), thus Lemma \ref{lem:grad_mu_closed} ensures
that \(\nabla_\mu f_n-\nabla_\mu g_n=\nabla_\mu(f_n-g_n)\) converges
to \(0\) in \(\Gamma^1_\mu(T_\mu)\). Moreover, if
\(f\in C^\infty_c(\R^d)\), then by taking the constant sequence
\(f_n\equiv f\) we see that \(\tilde\nabla_\mu f=\nabla_\mu f\).
Then we can omit the tilde from our notation and obtain an extension
\(\nabla_\mu\colon W^{1,1}(\R^d,\mu)\to\Gamma^1_\mu(T_\mu)\) of the
tangential gradient. Linearity readily follows from the fact that
\(\nabla_\mu\) is linear when restricted to \(C^\infty_c(\R^d)\).
Uniqueness is granted by its very construction. Moreover, observe that if
\(f\in W^{1,1}(\R^d,\mu)\) and \((f_n)_{n\in\N}\subseteq C^\infty_c(\R^d)\)
satisfy \(f_n\to f\) in \(W^{1,1}(\R^d,\mu)\), then we have both
\(\|f_n\|_{W^{1,1}_\mu}\to\|f\|_{W^{1,1}_\mu}\) and
\[
\|f_n\|_{W^{1,1}_\mu}=\|f_n\|_{L^1_\mu(\R^d)}+
\|\nabla_\mu f_n\|_{\Gamma^1_\mu(T_\mu)}\to
\|f\|_{L^1_\mu(\R^d)}+\|\nabla_\mu f\|_{\Gamma^1_\mu(T_\mu)},
\]
whence \eqref{eq:formula_W11_norm} follows. Finally,
it remains to show that \(\nabla_\mu\) is a closed operator,
namely, that if \((f_n)_{n\in\N}\subseteq W^{1,1}(\R^d,\mu)\)
satisfies \(f_n\rightharpoonup f\in L^1_\mu(\R^d)\) and
\(\nabla_\mu f_n\rightharpoonup v\in\Gamma^1_\mu(T_\mu)\),
then \(f\in W^{1,1}(\R^d,\mu)\) and \(\nabla_\mu f=v\).
Given any \(n\in\N\), we can find \(g_n\in C^\infty_c(\R^d)\)
such that \(\|g_n-f_n\|_{L^1_\mu(\R^d)}\leq 1/n\) and
\(\|\nabla_\mu g_n-\nabla_\mu f_n\|_{\Gamma^1_\mu(T_\mu)}\leq 1/n\).
In particular, \(g_n\rightharpoonup f\) in \(L^1_\mu(\R^d)\)
and \(\nabla_\mu g_n\rightharpoonup v\) in \(\Gamma^1_\mu(T_\mu)\).
Thanks to Mazur lemma (item ii) of Proposition \ref{prop:FA_facts})
and the linearity of \(\nabla_\mu\),
we may assume (possibly replacing the \(g_n\)'s by their
convex combinations) that \(g_n\to f\) in \(L^1_\mu(\R^d)\) and
\(\nabla_\mu g_n\to v\) in \(\Gamma^1_\mu(T_\mu)\). This implies
that \((g_n)_{n\in\N}\) is Cauchy in \(W^{1,1}(\R^d,\mu)\) and
that its \(L^1_\mu(\R^d)\)-limit coincides with \(f\), so that
\(f\in W^{1,1}(\R^d,\mu)\) and \(\nabla_\mu f=v\). Therefore,
the statement is achieved.
\end{proof}

We conclude the current section by proving that the (extended)
tangential gradient introduced in Proposition \ref{prop:ext_tg_grad} 
satisfies the following Leibniz rule:
\begin{lemma}[Leibniz rule for the tangential gradient]\label{lm:Leibniz_tan_gradient}
Let \(f,g\in W^{1,1}(\R^d,\mu)\cap L^\infty_\mu(\R^d)\). Then 
\begin{equation}\label{eq:Leibniz_tangential}
fg\in W^{1,1}(\R^d,\mu)\quad\text{ and }\quad
\nabla_\mu(fg)=g\nabla_\mu f+f\,\nabla_\mu g.
\end{equation}
\end{lemma}
\begin{proof}
We divide the proof into several steps:\\
{\color{blue}\textsc{Step 1.}}
We first prove that \eqref{eq:Leibniz_tangential} holds for any 
\(f,g\in W^{1,1}(\R^d,\mu)\cap L^\infty_\mu(\R^d)\) having bounded support. 
To verify this
claim, let us fix 
sequences \((f_n)_{n\in\N},(g_n)_{n\in\N}\subseteq C^\infty_c(\R^d)\)
such that \(f_n\to f\) and \(g_n\to g\) in \(W^{1,1}(\R^d,\mu)\).
Without loss of generality, we may assume that there exists \(C>0\) and 
a compact set \(K\subseteq \R^d\)
such that \(|f_n|,|g_n|\leq C\) and 
\({\rm supp}(f_n), {\rm supp}(g_n)\subseteq K\), for all \(n\in \N\).
Proposition \ref{prop:ext_tg_grad} ensures that
\(\nabla_\mu f_n\to\nabla_\mu f\) and \(\nabla_\mu g_n\to\nabla_\mu g\)
in \(\Gamma^1_\mu(T_\mu)\). Since also \(f_n\to f\) and \(g_n\to g\)
in \(L^1_\mu(\R^d)\), we know from Proposition \ref{prop:FA_facts} i)
that (up to a not relabelled subsequence) the convergence
\((f_n,\nabla_\mu f_n,g_n,\nabla_\mu g_n)\to(f,\nabla_\mu f,g,\nabla_\mu g)\)
is both dominated and in the pointwise \(\mu\)-a.e.\ sense. 
Now define \(h_n\coloneqq f_n g_n\in C^\infty_c(\R^d)\) and
\(v_n\coloneqq g_n\nabla_\mu f_n+f_n\nabla_\mu g_n\in\Gamma^1_\mu(T_\mu)\)
for every \(n\in\N\). Observe that \(\nabla_\mu h_n=v_n\),
as one can easily verify:
\[\begin{split}
\nabla_\mu h_n(x)&={\rm pr}_{T_\mu(x)}\big(\nabla(f_n g_n)(x)\big)
={\rm pr}_{T_\mu(x)}\big(g_n(x)\nabla f_n(x)+f_n(x)\nabla g_n(x)\big)\\
&=g_n(x){\rm pr}_{T_\mu(x)}\big(\nabla f_n(x)\big)+
f_n(x){\rm pr}_{T_\mu(x)}\big(\nabla g_n(x)\big)
=v_n(x),
\end{split}\]
for \(\mu\)-a.e.\ \(x\in\R^d\). Fixed a non-negative function
\(H\in L^1_\mu(\R^d)\) such that \(|\nabla_\mu f_n|,|\nabla_\mu g_n|
\leq H\) is satisfied \(\mu\)-a.e.\ for every \(n\in\N\), we can estimate
\(|h_n|\leq C^2\chi_K
\in L^1_\mu(\R^d)\) and 
\(|v_n|\leq 2C\chi_K H\in L^1_\mu(\R^d)\) in the
\(\mu\)-a.e.\ sense for every \(n\in\N\). By using the dominated
convergence theorem, we can finally conclude that
\[
h_n\to fg\,\text{ in }L^1_\mu(\R^d)\quad\text{ and }\quad\nabla_\mu h_n
=v_n\to g\nabla_\mu f+f\nabla_\mu g\,\text{ in }\Gamma^1_\mu(T_\mu).
\]
Therefore, Proposition \ref{prop:ext_tg_grad}
implies that \(fg\in W^{1,1}(\R^d,\mu)\) and
\(\nabla_\mu(fg)=g\nabla_\mu f+f\nabla_\mu g\).\\
{\color{blue}\textsc{Step 2.}} We next show that \eqref{eq:Leibniz_tangential} holds 
for \(f\in W^{1,1}(\R^d,\mu)\cap L^{\infty}_\mu(\R^d)\) and \(\eta\in C^{\infty}_c(\R^d)\).
Choosing a sequence of smooth functions \((f_n)_n\subseteq C^{\infty}_c(\R^d)\) such that 
\(f_n\to f\) in \(L^1_\mu(\R^d)\) and \(\nabla_\mu f_n\to \nabla_\mu f\) in 
\(\Gamma^1_\mu(T_\mu)\), we have that 
\(\nabla_\mu(f_n\eta)=\eta\nabla_\mu f_n+f_n\nabla_\mu \eta\) 
holds for every \(n\in \N\), and consequently 
\[
\int_{\R^d}\big|\nabla_\mu(f_n\eta)-(\eta\nabla_\mu f+f\nabla_\mu\eta)\big|\,\d\mu 
\leq \int_{\R^d}|\eta||(\nabla_\mu f_n-\nabla_\mu f)|\,\d\mu 
+\int_{\R^d}|(f_n-f)||\nabla_\mu\eta|\,\d\mu.
\]
Therefore, by passing to the limit as \(n\to \infty\), 
we get \eqref{eq:Leibniz_tangential}.\\
{\color{blue}\textsc{Step 3.}} We finally prove \eqref{eq:Leibniz_tangential} for any 
\(f,g\in W^{1,1}(\R^d,\mu)\cap L^{\infty}_{\mu}(\R^d)\). To this aim, fix a sequence 
\((\eta_n)_n\subseteq C^{\infty}_c(\R^d)\) of cut-off functions, \emph{i.e.}, 
\(0\leq \eta_n\leq 1\),
\(\eta_n=1\) on \(B_n(0)\) and \(|\nabla \eta_n|\leq 1\),
for every \(n\in \N\). 
Then for every \(n\in \N\), call \(f_n\coloneqq \eta_nf\) and \(g_n\coloneqq \eta_ng\) 
and observe that \(f_n,g_n\) are compactly supported functions belonging to  
\(W^{1,1}(\R^d,\mu)\cap L^{\infty}_{\mu}(\R^d)\) by \textsc{Step 2}. Thus, 
by \textsc{Step 1} we have that 
\(f_ng_n\in W^{1,1}(\R^d,\mu)\)
and that 
\[
\begin{split}
\nabla_\mu(f_ng_n)=&f_n\nabla_\mu g_n+g_n\nabla_\mu f_n
=\eta_n f\nabla_\mu(\eta_n g)+\eta_n g\nabla_\mu(\eta_n f)\\
=&\underset{\rm A}
{\underbrace{\big(\eta_n^2f\nabla_\mu g+\eta_n fg\nabla_\mu \eta_n \big)}}
+\underset{\rm B}{\underbrace{\big(\eta_n^2g\nabla_\mu f+\eta_ngf\nabla_\mu\eta_n\big)}}
\end{split}
\]
holds \(\mu\)-a.e.\ in \(\R^d\). Clearly, the sequence 
\((f_ng_n)_n\) converges to \(fg\) strongly in \(L^1_\mu(\R^d)\).
In order to prove \eqref{eq:Leibniz_tangential}, we will show that
\(\nabla_\mu(f_n g_n)\) converges to 
\(f\nabla_\mu g+g\nabla_\mu f\) strongly in \(L^1_\mu(\R^d)\).
We prove that the part of \(\nabla_\mu(f_n g_n)\) denoted by A above converges to 
\(f\nabla_\mu g\) strongly in \(L^1_\mu(\R^d)\). Indeed, it holds that
\[
\int_{\R^d} |\eta_n^2f\nabla_\mu g-f\nabla_\mu g|\,\d\mu\leq 
\int_{\R^d}|f||\nabla_\mu g|(1-\eta_n^2)\,\d\mu
\leq \int_{B_n^c(0)}|f||\nabla_\mu g|\,\d\mu
 \overset{n\to \infty}{\longrightarrow} 0
\]
and that 
\[
\int_{\R^d}|\eta_n f g\nabla_\mu\eta_n|\,\d\mu
\leq \int_{B_n^c(0)}|fg|\,\d\mu \overset{n\to \infty}{\longrightarrow} 0,
\]
proving the claim. Similarly, one can show that the part B converges to 
\(g\nabla_\mu f\) strongly in \(L^1_\mu(\R^d)\), concluding the proof. 
\end{proof}
\subsection{\texorpdfstring{\(W^{1,1}\)}{W11} space via relaxed slope} 
The relaxation type approach to the definition of \(W^{1,1}\) space proposed in
\cite{DMPhD}
is based on the concept of
the `relaxed slope':

\begin{definition}[Relaxed slope]
Let \(f\in L^1_\mu(\R^d)\). A non-negative function \(G\in L^1_\mu(\R^d)\)
is said to be a \emph{relaxed slope} of \(f\) if there exists a sequence 
\((f_n)_n\subseteq \LIP_c(\R^d)\) such that 
\(f_n\to f\) in \(L^1_\mu(\R^d)\) and \({\rm lip}_a(f_n)\rightharpoonup G'\) 
weakly in \(L^1_\mu(\R^d)\), for some \(G'\in L^1_\mu(\R^d)\) with 
\(G'\leq G\) \(\mu\)-a.e.\ in \(\R^d\).
We denote the set of all relaxed slopes of \(f\) by \({\rm RS}(f)\).
\end{definition}
\begin{definition}[\(W^{1,1}\) space via relaxation] 
We say that a function \(f\in L^1_\mu(\R^d)\) belongs to 
the space \(W^{1,1}_{\rm Lip}(\R^d,\mu)\) if \({\rm RS}(f)\neq \emptyset\).
The minimal element (in the \(\mu\)-a.e.\ sense) of \({\rm RS}(f)\) 
will be denoted by \(|\nabla f|_{rs}\) and called the \emph{minimal relaxed slope} 
of \(f\).
\end{definition}
\begin{remark}\label{rmk:approx_lip_a_strong}{\rm
Given any \(f\in W^{1,1}_{\rm Lip}(\R^d,\mu)\),
there exist \((f_n)_{n\in\N}\subseteq\LIP_c(\R^d)\) and
\(H\in L^1_\mu(\R^d)\) such that \(f_n\to f\) strongly in
\(L^1_\mu(\R^d)\), \(\lip_a(f_n)\rightharpoonup|\nabla f|_{rs}\) weakly
in \(L^1_\mu(\R^d)\), and \(\lip_a(f_n)\leq H\) \(\mu\)-a.e.\ for
every \(n\in\N\). Indeed, by exploiting the minimality of
\(|\nabla f|_{rs}\) we can find \((g_n)_{n\in\N}\subseteq\LIP_c(\R^d)\)
such that \(g_n\to f\) strongly in \(L^1_\mu(\R^d)\) and
\(\lip_a(g_n)\rightharpoonup|\nabla f|_{rs}\) weakly in
\(L^1_\mu(\R^d)\). By Proposition \ref{prop:FA_facts} ii),
we can find \((\alpha_i^n)_{i=n}^{N_n}\subseteq[0,1]\)
with \(\sum_{i=n}^{N_n}\alpha_i^n=1\) and
\(\sum_{i=n}^{N_n}\alpha_i^n\lip_a(g_i)\to|\nabla f|_{rs}\)
strongly in \(L^1_\mu(\R^d)\) as \(n\to\infty\).
By Proposition \ref{prop:FA_facts} i), we know that there exists
\(H\in L^1_\mu(\R^d)\) such that (up to a not relabeled subsequence
in \(n\)) it holds \(\sum_{i=n}^{N_n}\alpha_i^n\lip_a(g_i)\leq H\)
\(\mu\)-a.e.\ for every \(n\in\N\). Now we define
\(f_n\coloneqq\sum_{i=n}^{N_n}\alpha_i^n g_i\in\LIP_c(\R^d)\)
for every \(n\in\N\). Notice that \(f_n\to f\) strongly in
\(L^1_\mu(\R^d)\) and
\[
\lip_a(f_n)=\lip_a\bigg(\sum_{i=n}^{N_n}\alpha_i^n g_i\bigg)
\leq\sum_{i=n}^{N_n}\alpha_i^n\lip_a(g_i)\leq H,\quad\text{ in the }
\mu\text{-a.e.\ sense},
\]
for every \(n\in\N\). By Proposition \ref{prop:FA_facts} iii),
we can find \(G\in L^1_\mu(\R^d)\) such that \(G\leq|\nabla f|_{rs}\)
\(\mu\)-a.e.\ and (up to a further subsequence in \(n\)) it holds
\(\lip_a(f_n)\rightharpoonup G\) weakly in \(L^1_\mu(\R^d)\).
Finally, the minimality of \(|\nabla f|_{rs}\) ensures that
\(G=|\nabla f|_{rs}\), thus accordingly the claim is proved.
\fr}\end{remark}
\subsection{\texorpdfstring{\(W^{1,1}\)}{W11}
space via tangential relaxed slope}
We introduce here an auxiliary notion of \(W^{1,1}\) space, which is
intermediate between the approaches \(W^{1,1}(\R^d,\mu)\) and
\(W^{1,1}_{\rm Lip}(\R^d,\mu)\), as it is based upon the relaxation
of the modulus of the tangential gradient \(\nabla_\mu\). Its
equivalence with \(W^{1,1}(\R^d,\mu)\) will be proved in Theorem
\ref{thm:equiv_char_W11}. As a consequence of this characterization,
we will show in Proposition \ref{prop:Lipc_W1,1} that, as one might expect, the
compactly-supported Lipschitz functions belong to \(W^{1,1}(\R^d,\mu)\).
\begin{definition}[Tangential relaxed slope]\label{def:TRS}
Given any \(f\in L^1_\mu(\R^d)\), we say that \(G\in L^1_\mu(\R^d)\)
is a \emph{tangential relaxed slope} of \(f\) provided there exists
a sequence  \((f_n)_n\subseteq C^\infty_c(\R^d)\) converging to \(f\)
in \(L^1_\mu(\R^d)\) and such that \(|\nabla_\mu f_n|\rightharpoonup G'\)
weakly in \(L^1_\mu(\R^d)\), for some function \(G'\in L^1_\mu(\R^d)\)
with  \(G'\leq G\) \(\mu\)-a.e.\ in \(\R^d\). We denote by
\({\rm TRS}(f)\) the family of all tangential relaxed slopes of \(f\).
\end{definition}
\begin{lemma}
Let \(f\in L^1_\mu(\R^d)\) be such that \({\rm TRS}(f)\neq\emptyset\).
Then the set \({\rm TRS}(f)\) is a closed convex sublattice of
\(L^1_\mu(\R^d)\). In particular, it admits a unique \(\mu\)-a.e.\ minimal
element \(G_f\in{\rm TRS}(f)\), namely \(G_f\leq G\) holds \(\mu\)-a.e.\ for
every \(G\in{\rm TRS}(f)\).
\end{lemma}
\begin{proof}
\ \\
{\color{blue}\textsc{Closure.}} Let \((G_i)_i\subseteq{\rm TRS}(f)\)
satisfy \(G_i\to G\) strongly in \(L^1_\mu(\R^d)\) for some
\(G\in L^1_\mu(\R^d)\). We aim to show that \(G\in{\rm TRS}(f)\).
For any \(i\in\N\), we can find \(G'_i\leq G_i\) and
\((f^i_n)_n\subseteq C^\infty_c(\R^d)\) such that \(f^i_n\to f\)
strongly in \(L^1_\mu(\R^d)\) as \(n\to\infty\) and
\(|\nabla_\mu f^i_n|\rightharpoonup G'_i\) weakly in \(L^1_\mu(\R^d)\)
as \(n\to\infty\). Since \(G_i\to G\) in \(L^1_\mu(\R^d)\), we know
from Proposition \ref{prop:FA_facts} i) that (up to a not relabeled
subsequence) it holds \(G'_i\leq G_i\leq H\) \(\mu\)-a.e.\ for all
\(i\in\N\), for some \(H\in L^1_\mu(\R^d)\). Hence,
Proposition \ref{prop:FA_facts} iii) ensures that (up to a further
subsequence) it holds \(G'_i\rightharpoonup G'\) weakly in
\(L^1_\mu(\R^d)\), for some \(G'\in L^1_\mu(\R^d)\). Since
\(G'_i\leq G_i\) \(\mu\)-a.e.\ for all \(i\in\N\), we deduce that
\(G'\leq G\) \(\mu\)-a.e.. We now perform a diagonalization argument:
for any \(i\in\N\) we can find \(n(i)\in\N\) such that the elements
\(f_i\coloneqq f^i_{n(i)}\) satisfy \(f_i\to f\) strongly in
\(L^1_\mu(\R^d)\) and \(|\nabla_\mu f_i|\rightharpoonup G'\) weakly
in \(L^1_\mu(\R^d)\). This shows that \(G\in{\rm TRS}(f)\), as desired.\\
{\color{blue}\textsc{Convexity.}} Fix any \(G,H\in{\rm TRS}(f)\) and
\(\lambda\in[0,1]\). Then there exist \(G'\leq G\), \(H'\leq H\),
and \((f_n)_n,(g_n)_n\subseteq C^\infty_c(\R^d)\) such that
\(f_n\to f\), \(g_n\to f\), \(|\nabla_\mu f_n|\rightharpoonup G'\),
and \(|\nabla_\mu g_n|\rightharpoonup H'\) in \(L^1_\mu(\R^d)\).
Proposition \ref{prop:FA_facts} yields the existence of some coefficients
\((\alpha^n_i)_{i=n}^{N_n},(\beta^n_i)_{i=n}^{M_n}\subseteq[0,1]\) with
\(\sum_{i=n}^{N_n}\alpha^n_i=\sum_{i=n}^{M_n}\beta^n_i=1\) such that
\(\big(\sum_{i=n}^{N_n}\alpha^n_i|\nabla_\mu f_i|\big)_n\) and
\(\big(\sum_{i=n}^{M_n}\beta^n_i|\nabla_\mu g_i|\big)_n\) are dominated
and converge strongly in \(L^1_\mu(\R^d)\) to \(G'\) and \(H'\),
respectively. For any \(n\in\N\) we define
\[
h_n\coloneqq\lambda\sum_{i=n}^{N_n}\alpha^n_i f_i+(1-\lambda)
\sum_{i=n}^{M_n}\beta^n_i g_i\in C^\infty_c(\R^d).
\]
Observe that \(h_n\to f\) strongly in \(L^1_\mu(\R^d)\).
Moreover, from the inequality
\[
|\nabla_\mu h_n|\leq\lambda\sum_{i=n}^{N_n}\alpha^n_i|\nabla_\mu f_i|
+(1-\lambda)\sum_{i=n}^{M_n}\beta^n_i|\nabla_\mu g_i|
\]
we deduce that \(\big(|\nabla_\mu h_n|\big)_n\) is dominated,
thus (by Proposition \ref{prop:FA_facts} iii) and up to subsequence) it
holds \(|\nabla_\mu h_n|\rightharpoonup L'\) weakly in \(L^1_\mu(\R^d)\),
for some \(L'\leq\lambda G'+(1-\lambda)H'\leq\lambda G+(1-\lambda)H\).
This shows that \(\lambda G+(1-\lambda)H\in{\rm TRS}(f)\), thus proving
the convexity of the set \({\rm TRS}(f)\).\\
{\color{blue}\textsc{Lattice property.}} We aim to show that,
given any \(G,H\in{\rm TRS}(f)\), it holds \(G\vee H\in{\rm TRS}(f)\)
and \(G\wedge H\in{\rm TRS}(f)\). The former is trivial, so let us
focus on the latter. Define \(E\coloneqq\{G\leq H\}\). By convolution,
we can find a sequence \((\eta_j)_j\subseteq C^\infty_c(\R^d)\) with
\(0\leq\eta_j\leq 1\) that weakly\(^*\) converges to \(\chi_E\) in
\(L^\infty_\mu(\R^d)\). In particular, \(\eta_j G+(1-\eta_j)H
\rightharpoonup\chi_E G+\chi_{E^c}H=G\wedge H\) weakly in
\(L^1_\mu(\R^d)\). The set \({\rm TRS}(f)\subseteq L^1_\mu(\R^d)\)
is weakly closed (as it is strongly closed and convex), thus in order
to prove that \(G\wedge H\in{\rm TRS}(f)\) it suffices to show
that \(\eta_j G+(1-\eta_j)H\in{\rm TRS}(f)\) for all \(j\in\N\).
To this aim, pick \(G'\leq G\), \(H'\leq H\), and
\((f_n)_n,(g_n)_n\subseteq C^\infty_c(\R^d)\) such that
\(f_n\to f\), \(g_n\to f\), \(|\nabla_\mu f_n|\rightharpoonup G'\),
and \(|\nabla_\mu g_n|\rightharpoonup H'\) in \(L^1_\mu(\R^d)\).
Thanks to Proposition \ref{prop:FA_facts}, for any \(n\in\N\)
we can find coefficients
\((\alpha^n_i)_{i=n}^{N_n},(\beta^n_i)_{i=n}^{M_n}\subseteq[0,1]\) with
\(\sum_{i=n}^{N_n}\alpha^n_i=\sum_{i=n}^{M_n}\beta^n_i=1\) such that
the sequences \(\big(\sum_{i=n}^{N_n}\alpha^n_i|\nabla_\mu f_i|\big)_n\)
and \(\big(\sum_{i=n}^{M_n}\beta^n_i|\nabla_\mu g_i|\big)_n\) are
dominated and converge strongly in \(L^1_\mu(\R^d)\) to \(G'\) and
\(H'\), respectively. Now fix \(j\in\N\) and define
\(h^j_n\coloneqq\eta_j\sum_{i=n}^{N_n}f_i+(1-\eta_j)
\sum_{i=n}^{M_n}g_i\in C^\infty_c(\R^d)\) for every \(n\in\N\).
Observe that \(h^j_n\to f\) strongly in \(L^1_\mu(\R^d)\)
as \(n\to\infty\). For any \(n\in\N\) we have that
\[
|\nabla_\mu h^j_n|\leq\eta_j\sum_{i=n}^{N_n}\alpha^n_i|\nabla_\mu f_i|
+(1-\eta_j)\sum_{i=n}^{M_n}\beta^n_i|\nabla_\mu g_i|,
\]
thus in particular \(\big(|\nabla_\mu h^j_n|\big)_n\) is dominated.
Therefore, Proposition \ref{prop:FA_facts} iii) ensures that (up to
a not relabeled subsequence in \(n\)) it holds \(|\nabla_\mu h^j_n|
\rightharpoonup L'\) weakly in \(L^1_\mu(\R^d)\) as \(n\to\infty\), for
some function \(L'\leq\eta_j G'+(1-\eta_j)H'\leq\eta_j G+(1-\eta_j)H\).
This yields \(\eta_j G+(1-\eta_j)H\in{\rm TRS}(f)\).
\end{proof}
\begin{theorem}\label{thm:equiv_char_W11}
Let \(f\in L^1_\mu(\R^d)\) be given. Then \(f\in W^{1,1}(\R^d,\mu)\) if
and only if \({\rm TRS}(f)\neq \emptyset\). In this case, the function
\(|\nabla_\mu f|\) coincides with the \(\mu\)-a.e.\ minimal element
of \({\rm TRS}(f)\).
\end{theorem}
\begin{proof}
First, we aim to show that if \(f\in W^{1,1}(\R^d,\mu)\),
then \({\rm TRS}(f)\neq\emptyset\) and \(G_f\leq|\nabla_\mu f|\)
\(\mu\)-a.e., where \(G_f\) stands for the minimal element of
\({\rm TRS}(f)\). Pick \((f_n)_n\subseteq C^\infty_c(\R^d)\) such that
\(f_n\to f\) and \(\nabla_\mu f_n\to\nabla_\mu f\) strongly in
\(L^1_\mu(\R^d)\) and \(\Gamma^1_\mu(T_\mu)\), respectively. In particular,
 \(|\nabla_\mu f_n|\to|\nabla_\mu f|\) strongly in
\(L^1_\mu(\R^d)\), whence it follows that \(|\nabla_\mu f|\in{\rm TRS}(f)\),
thus \(G_f\leq|\nabla_\mu f|\) in the \(\mu\)-a.e.\ sense.

Conversely, let us show that if \({\rm TRS}(f)\neq\emptyset\),
then \(f\in W^{1,1}(\R^d,\mu)\) and \(|\nabla_\mu f|\leq G\)
\(\mu\)-a.e.\ for every \(G\in{\rm TRS}(f)\). There exist
\(G'\leq G\) and \((f_n)_n\subseteq C^\infty_c(\R^d)\) such
that \(f_n\to f\) and \(|\nabla_\mu f_n|\rightharpoonup G'\)
in \(L^1_\mu(\R^d)\). Proposition \ref{prop:FA_facts} yields the
existence of \((\alpha^n_i)_{i=n}^{N_n}\subseteq[0,1]\) with
\(\sum_{i=n}^{N_n}\alpha^n_i=1\) such that the sequence
\(\big(\sum_{i=n}^{N_n}\alpha^n_i|\nabla_\mu f_i|\big)_n\) is
dominated and converges (both strongly in \(L^1_\mu(\R^d)\) and in
the pointwise \(\mu\)-a.e.\ sense) to \(G'\) as \(n\to\infty\). Define
\[
g_n\coloneqq\sum_{i=n}^{N_n}\alpha^n_i f_i\in C^\infty_c(\R^d),
\quad\text{ for every }n\in\N.
\]
Then \(g_n\to f\) strongly in \(L^1_\mu(\R^d)\). Moreover, from
the inequality \(|\nabla_\mu g_n|\leq\sum_{i=n}^{N_n}\alpha^n_i
|\nabla_\mu f_i|\) we deduce that the sequence \((\nabla_\mu g_n)_n\)
is dominated. Hence, by applying Proposition \ref{prop:FA_facts} iii)
we obtain that there exists a vector field \(v\in\Gamma^1_\mu(T_\mu)\)
such that (up to a subsequence in \(n\)) it holds
\(\nabla_\mu g_n\rightharpoonup v\) weakly in \(\Gamma^1_\mu(T_\mu)\).
Lemma \ref{lem:grad_mu_closed} ensures that \(f\in W^{1,1}(\R^d,\mu)\)
and \(v=\nabla_\mu f\). Finally,
let us show that \(|\nabla_\mu f|\leq G\) \(\mu\)-a.e.\ in \(\R^d\).
Given any \(v\in \R^d\) with \(|v|\leq 1\) and any 
\(0\leq h\in L^\infty_\mu(\R^d)\) it holds that 
\[
\int_{\R^d}h\, v\cdot \nabla_\mu f\,\d\mu
=\lim_{n\to \infty}\int_{\R^d}h\,v\cdot \nabla_\mu g_n\,\d\mu
\leq \lim_{n\to \infty}\int_{\R^d}h|\nabla_\mu g_n|\,\d\mu
\leq\int_{\R^n}hG'\,\d\mu\leq \int_{\R^d}hG\,\d\mu.
\] 
By the arbitrariness of \(h\), we deduce that 
\(v\cdot \nabla_\mu f\leq G\) holds \(\mu\)-a.e.\ in \(\R^d\). Then, 
we conclude that 
\(
|\nabla_\mu f|(x)=
\sup\{ v\cdot \nabla_\mu f(x):\, v\in \R^d,\, |v|\leq 1\}\leq G
\) for \(\mu\)-a.e.\  \(x\in\R^d\).
Therefore, the statement is achieved.
\end{proof}
\begin{proposition}\label{prop:Lipc_W1,1}
Let \(f\in \LIP_c(\R^d)\). Then \(f\in W^{1,1}(\R^d,\mu)\) and 
\(|\nabla_\mu f|\leq {\rm lip}_a(f)\) holds \(\mu\)-a.e..
\end{proposition}
\begin{proof}
Denote by \(K\) the closed \(1\)-neighbourhood of \({\rm supp}(f)\).
Fix a sequence \((\varepsilon_n)_n\subseteq(0,1)\) such that
\(\varepsilon_n\to 0\) and define \(f_n\coloneqq\rho_{\varepsilon_n}*f
\in C^\infty_c(\R^d)\) for every \(n\in\N\). Notice that each
\({\rm supp}(f_n)\) is contained in \(K\). Thanks to
\eqref{eq:approx_lip_via_smooth_1}, we see that
\(|f_n|\leq\|f\|_{L^\infty_\mu(\R^d)}\chi_K\) and \(f_n(x)\to f(x)\)
for all \(x\in\R^d\), 
thus by using the dominated convergence theorem we obtain that
\(f_n\to f\) in \(L^1_\mu(\R^d)\). Moreover,
\[
|\nabla_\mu f_n|\leq|\nabla f_n|
\overset{\eqref{eq:approx_lip_via_smooth_2}}\leq
{\rm Lip}\big(f;B_{\varepsilon_n}(\cdot)\big)\leq
{\rm Lip}(f)\chi_K\in L^1_\mu(\R^d),\quad\text{ for every }n\in\N,
\]
thus in particular \(\big(|\nabla_\mu f_n|\big)_n\) is dominated.
By Proposition \ref{prop:FA_facts} iii) we get the existence of
a function \(G\in L^1_\mu(\R^d)\) such that (up to a subsequence) it
holds \(|\nabla_\mu f_n|\rightharpoonup G\) weakly in \(L^1_\mu(\R^d)\).
Hence, we conclude that for every \(0\leq h\in L^\infty_\mu(\R^d)\)
it holds that
\begin{equation}\label{eq:aux}
\begin{split}
\int_{\R^d}hG\,\d\mu
=\lim_{n\to\infty}\int_{\R^d}h|\nabla_\mu f_n|\,\d\mu
\overset{\phantom{\overset{\eqref{eq:approx_lip_via_smooth_2}}
\leq}}\leq &\lim_{n\to\infty}\int_{\R^d}h|\nabla f_n|\,\d\mu\\
\overset{\eqref{eq:approx_lip_via_smooth_2}}
\leq &\lim_{n\to\infty}\int_{\R^d}h\,{\rm Lip}\big(f;B_{\varepsilon_n}
(\cdot)\big)\,\d\mu
=\int_{\R^d}h\,\lip_a(f)\,\d\mu,
\end{split}
\end{equation}
where, in order to get the last equality above, we have used a simple 
technical result given in Lemma \ref{lem:conv_to_lip_a} below.
It follows from \eqref{eq:aux}
that \(G\leq {\rm lip}_a(f)\) holds \(\mu\)-a.e.\ and thus 
that \(\lip_a(f)\in{\rm TRS}(f)\). Accordingly,
\(f\in W^{1,1}(\R^d,\mu)\) and \(|\nabla_\mu f|\leq\lip_a(f)\)
holds \(\mu\)-a.e.\ by Theorem \ref{thm:equiv_char_W11}. The
proof is complete.
\end{proof}
In the proof of the lemma above  the following
easy technical lemma has been used. It will be useful also later on, in the 
proof of 
Theorem \ref{thm:relation_W11}.

\begin{lemma}\label{lem:conv_to_lip_a}
Let \(f\in\LIP_c(\R^d)\) be given. Then it holds that
\begin{equation}\label{eq:conv_to_lip_a}
{\rm Lip}\big(f;B_r(\cdot)\big)\to\lip_a(f),
\quad\text{ strongly in }L^1_\mu(\R^d)\text{ as }r\searrow 0.
\end{equation}
\end{lemma}
\begin{proof}
Call \(K\) the closed \(1\)-neighbourhood of \({\rm supp}(f)\).
Observe that for any \(r\in(0,1)\) we have that
\({\rm Lip}\big(f;B_r(\cdot)\big)\leq{\rm Lip}(f)\chi_K\in L^1_\mu(\R^d)\)
on \(\R^d\). Recalling that
\(\lim_{r\searrow 0}{\rm Lip}\big(f;B_r(x)\big)=\lip_a(f)(x)\)
for every \(x\in\R^d\) by the very definition of \({\rm lip}_a(f)\),
by applying the dominated convergence theorem we conclude that
\eqref{eq:conv_to_lip_a} is verified, as desired.
\end{proof}
\section{Relation between BV and \texorpdfstring{\(W^{1,1}\)}{W11} spaces}
\subsection{The total variation measure}\label{ssec:total_var_meas}
Given that the relaxation-type approach to the BV space presented in 
Subsection \ref{ssec:BV_relaxed} comes from the general metric measure space setting, 
one has to use Lipschitz functions in the relaxation process. When we stick to 
the specific case of the weighted Euclidean space, one would expect 
that Lipschitz functions can be replaced by smooth ones. Indeed,
in the next results we confirm it. In this subsection we will use the notation 
\(\LIP_{bs}(\X)\) to denote the space of boundedly-supported Lipschitz functions on 
a given metric space \((\X, \sf d)\). We start with two preparatory lemmata:
\begin{lemma}\label{lem:localize_tv}
Let \(f\in{\rm BV}_{\rm Lip}(\R^d,\mu)\) be given. Let \(\Omega\subseteq\R^d\)
be an open set with \(|D_\mu f|_{\rm Lip}(\partial\Omega)=0\). Then
\(f|_{\bar\Omega}\in{\rm BV}_{\rm Lip}(\bar\Omega,\mu|_{\bar\Omega})\)
and \(\big|D_{\mu|_{\bar\Omega}}(f|_{\bar\Omega})\big|_
{\rm Lip}(\bar\Omega)=|D_\mu f|_{\rm Lip}(\Omega)\).
\end{lemma}
\begin{proof}
For brevity, call \(C\coloneqq\bar\Omega\) and \(g\coloneqq f|_C\).
Given that
\[
|D_\mu f|_{\rm Lip}(\Omega)=|D_\mu f|(C)=\inf_U|D_\mu f|_{\rm Lip}(U),\]
where the infimum is among all open sets \(U\subseteq\R^d\) containing \(C\),
for any \(\varepsilon>0\) we can find an open set \(U\subseteq\R^d\)
with \(C\subseteq U\) and a sequence
\((f_n)_n\subseteq\LIP_{loc}(U)\cap L^1_\mu(U)\)
such that \(f_n\to f|_U\) in \(L^1_\mu(U)\) and
\(\limi_n\int_U\lip_a(f_n)\,\d\mu\leq|D_\mu f|_{\rm Lip}(\Omega)+\varepsilon\).
Then \(g_n\coloneqq f_n|_C\in\LIP_{loc}(C)\cap L^1_\mu(C)\)
satisfies \(g_n\to g\) and
\(\limi_n\int_C\lip_a(g_n)\,\d\mu\leq\limi_n\int_U\lip_a(f_n)\,\d\mu
\leq|D_\mu f|_{\rm Lip}(\Omega)+\varepsilon\). Hence, we have
\(g\in{\rm BV}_{\rm Lip}(C,\mu|_C)\)
and \(|D_{\mu|_C}g|_{\rm Lip}(C)\leq|D_\mu f|_{\rm Lip}(\Omega)\). Conversely, 
we can find a sequence
\((g'_n)_n\subseteq\LIP_{bs}(C)\) such that \(g'_n\to g\) in \(L^1_\mu(C)\)
and \(\int_C\lip_a(g'_n)\,\d\mu\to|D_{\mu|_C}g|_{\rm Lip}(C)\), thus
the functions \(f'_n\coloneqq g'_n|_\Omega\in\LIP_{loc}(\Omega)\cap
L^1_\mu(\Omega)\) satisfy \(f'_n\to f|_\Omega\) in \(L^1_\mu(\Omega)\)
and
\[
\limi_{n\to\infty}\int_\Omega\lip_a(f'_n)\,\d\mu\leq
\lim_{n\to\infty}\int_C\lip_a(g'_n)\,\d\mu=|D_{\mu|_C}g|_{\rm Lip}(C),
\]
whence it follows that
\(|D_\mu f|_{\rm Lip}(\Omega)\leq|D_{\mu|_C}g|_{\rm Lip}(C)\).
Therefore, the statement is achieved.
\end{proof}
\begin{lemma}\label{lem:good_open_sets}
Let \(f\in{\rm BV}_{\rm Lip}(\R^d,\mu)\) be given. Define
\begin{equation}\label{eq:def_O_f}
\mathcal O_f\coloneqq\big\{\Omega\subseteq\R^d\text{ open}\;\big|\;
\mu(\Omega)<+\infty,\,\mu(\partial\Omega)=|D_\mu f|_{\rm Lip}(\partial\Omega)=0\big\}.
\end{equation}
Then for any \(\sigma\)-finite Borel measure \(\nu\) on \(\R^d\) and
for any compact set \(K\subseteq\R^d\) it holds that
\begin{equation}\label{eq:good_open_sets}
\nu(K)=\inf\big\{\nu(\Omega)\;\big|\;\Omega\in\mathcal O_f,\,
K\subseteq\Omega\big\}.
\end{equation}
\end{lemma}
\begin{proof}
For any \(r>0\), denote by \(\Omega_r\) the open \(r\)-neighbourhood
of \(K\). Since the sets \(\{\Omega_r\}_{r>0}\) have pairwise disjoint
boundaries, we deduce that \(\Omega_r\in\mathcal O_f\) for a.e.\ \(r>0\).
In particular, we can find a sequence \(r_i\searrow 0\) such that
\((\Omega_{r_i})_i\subseteq\mathcal O_f\). Given that
\(K=\bigcap_i\Omega_{r_i}\), we conclude that \(\nu(K)=\lim_i
\nu(\Omega_{r_i})\), whence the claim \eqref{eq:good_open_sets} follows.
\end{proof}
Now, given a function
\(f\in{\rm BV}_{\rm Lip}(\R^d,\mu)\) and an open set \(\Omega\subseteq\R^d\),
we define
\begin{equation}\label{eq:BV_C_infty_open}
|D_\mu f|_{C^\infty}(\Omega)\coloneqq\inf\bigg\{\limi_{n\to\infty}
\int_\Omega|\nabla f_n|\,\d\mu\;\bigg|\;(f_n)_n\subseteq C^\infty(\Omega)
\cap L^1_\mu(\Omega),\,f_n\to f\text{ in }L^1_\mu(\Omega)\bigg\}.
\end{equation}
We can extend it via Carath\'{e}odory construction to a set-function
on all Borel sets, as follows:
\[
|D_\mu f|_{C^\infty}(B)\coloneqq\inf\big\{|D_\mu f|_{C^\infty}(\Omega)
\;\big|\;\Omega\subseteq\R^d\text{ open},\,B\subseteq\Omega\big\}.
\]
By suitably adapting the arguments in \cite[Lemma 4.4.2 and Lemma 4.4.3]{DMPhD},
one can show that \(|D_\mu f|_{C^\infty}\) is a finite Borel measure on \(\R^d\).
\begin{theorem}\label{thm:Lip_Cinfty}
Let \(f\in{\rm BV}_{\rm Lip}(\R^d,\mu)\). Then the measures
\(|D_\mu f|_{C^\infty}\) and \(|D_\mu f|_{\rm Lip}\) coincide.
\end{theorem}
\begin{proof}
It suffices to show that \(|D_\mu f|_{C^\infty}(K)=|D_\mu f|_{\rm Lip}(K)\) for every
\(K\subseteq\R^d\) compact. Thanks to Lemma \ref{lem:good_open_sets},
this is verified as soon as \(|D_\mu f|_{C^\infty}(\Omega)=|D_\mu f|_{\rm Lip}(\Omega)\)
for every \(\Omega\in\mathcal O_f\), where \(\mathcal O_f\) is defined
as in \eqref{eq:def_O_f}. Then let \(\Omega\in\mathcal O_f\) be fixed.
Since \(C^\infty(\Omega)\subseteq\LIP_{loc}(\Omega)\) and
\(|\nabla g|=\lip_a(g)\) for all \(g\in C^\infty(\Omega)\), we have that
\(|Df|_{\rm Lip}(\Omega)\leq|D_\mu f|_{C^\infty}(\Omega)\). To prove the converse
inequality, we apply Lemma \ref{lem:localize_tv}: given that
\(f|_{\bar\Omega}\in{\rm BV}_{\rm Lip}(\bar\Omega,\mu|_{\bar\Omega})\) and
\(\big|D_{\mu|_{\bar\Omega}}(f|_{\bar\Omega})\big|_{\rm Lip}(\bar\Omega)=|D_\mu f|_{\rm Lip}(\Omega)\), there exists a
sequence \((f_n)_n\subseteq\LIP_{bs}(\bar\Omega)\) such that
\(f_n\to f\) in \(L^1_\mu(\bar\Omega)\) and
\begin{equation}\label{eq:local_BV_Lip_Cinfty_aux}
\int_\Omega\lip_a(f_n)\,\d\mu=\int_{\bar\Omega}\lip_a(f_n)\,\d\mu\to
\big|D_{\mu|_{\bar\Omega}}(f|_{\bar\Omega})\big|_{\rm Lip}(\bar\Omega)
=|D_\mu f|_{\rm Lip}(\Omega),
\end{equation}
where the first identity is granted by the fact that \(\partial\Omega\)
is \(\mu\)-negligible. Now let \(n\in\N\) be fixed. Extend \(f_n\) to
some \({\rm Lip}(f_n)\)-Lipschitz function \(\bar f_n\colon\R^d\to\R\).
Define \(f_n^m\coloneqq\bar f_n*\rho_{1/m}\in C^\infty(\R^d)\) and
\(g_n^m\coloneqq f_n^m|_\Omega\in C^\infty(\Omega)\) for every \(m\in\N\).
Since \(|g_n^m-f_n|\leq{\rm Lip}(f_n)/m\) by
\eqref{eq:approx_lip_via_smooth_1}, we may estimate
\[
\int_\Omega|g_n^m|\,\d\mu\leq\int_\Omega|f_n|\,\d\mu
+\frac{{\rm Lip}(f_n)\mu(\Omega)}{m}<+\infty,
\]
so that \(g_n^m\in L^1_\mu(\Omega)\). By dominated convergence theorem,
we also obtain that \(g_n^m\to f_n\) in \(L^1_\mu(\Omega)\) as
\(m\to\infty\). Moreover, we have \(|\nabla g_n^m|(x)\leq
{\rm Lip}\big(\bar f_n;B_{2/m}(x)\big)\) for all \(m\in\N\)
and \(x\in\Omega\) by \eqref{eq:approx_lip_via_smooth_2},
so that \(|\nabla g_n^m|\leq{\rm Lip}(f_n)\chi_\Omega\in L^1_\mu(\Omega)\)
on \(\Omega\). An application of the reverse Fatou lemma yields
\[\begin{split}
\lims_{m\to\infty}\int_\Omega|\nabla g_n^m|\,\d\mu&\leq\lims_{m\to\infty}
\int_\Omega{\rm Lip}\big(\bar f_n;B_{2/m}(x)\big)\,\d\mu(x)\leq
\int_\Omega\lim_{m\to\infty}{\rm Lip}\big(\bar f_n;B_{2/m}(x)\big)\,\d\mu(x)\\
&=\int_\Omega\lip_a(f_n)\,\d\mu.
\end{split}\]
Hence, we can choose \(m_n\in\N\) such that the function
\(g_n\coloneqq g_n^{m_n}\) satisfies \(\int_\Omega|g_n-f_n|\,\d\mu\leq 1/n\)
and \(\int_\Omega|\nabla g_n|\,\d\mu\leq\int_\Omega\lip_a(f_n)\,\d\mu
+1/n\).
Then \(C^\infty(\Omega)\cap L^1_\mu(\Omega)\ni g_n\to f\)
in \(L^1_\mu(\Omega)\), so that accordingly
\[
|D_\mu f|_{C^\infty}(\Omega)\leq\limi_{n\to\infty}\int_\Omega|\nabla g_n|\,\d\mu
\leq\lim_{n\to\infty}\int_\Omega\lip_a(f_n)\,\d\mu
\overset{\eqref{eq:local_BV_Lip_Cinfty_aux}}=|D_\mu f|_{\rm Lip}(\Omega).
\]
All in all, we have proved that \(|D_\mu f|_{C^\infty}(\Omega)
=|D_\mu f|_{\rm Lip}(\Omega)\) for every \(\Omega\in\mathcal O_f\),
as desired.
\end{proof}
Note that it also follows from Theorem \ref{thm:Lip_Cinfty} that 
\begin{equation}\label{eq:BV_Cinfty}
{\rm BV}_{\rm Lip}(\R^d, \mu)={\rm BV}_{C^\infty}(\R^d, \mu)\coloneqq 
\big\{f\in L^1_\mu(\R^d):\, |D_\mu f|_{C^\infty}(\R^d)<+\infty\big\}.
\end{equation}  

Moreover, the total variation measure of the entire space can be recovered by using 
only compactly-supported smooth functions:
\begin{lemma}\label{lm:total_mass_smooth}
Let \(f\in {\rm BV}_{\rm Lip}(\R^d,\mu)\). Then
\begin{equation}\label{eq:total_mass_smooth}
|Df|_{\rm Lip}(\R^d)=\inf\; \Big\{\limi_{n\to \infty}\int_{\R^d}|\nabla f_n|\,
{\rm d}\mu:\;(f_n)_n\subseteq C^{\infty}_c(\R^d),\, 
f_n\to f\text{ \rm in }L^1_\mu(\R^d)\Big\}.
\end{equation}
\end{lemma}
\begin{proof}
Denote by \({\rm R}(f)\) the right-hand side of \eqref{eq:total_mass_smooth}.
Pick a sequence \((f_n)_n\subseteq \LIP_c(\R^d)\) 
converging to \(f\) in \(L^1_\mu(\R^d)\) (whose existence is guaranteed by 
the characterization of BV functions in item 2) of Definition \ref{def:BVmms}).
Fix $n\in \N$ and denote by 
\((f^m_n)_m\subseteq C^{\infty}_c(\R^d)\) the sequence satisfying
\begin{equation}\label{eq:sequence_1}
\big|f^m_n(x)-f_n(x)\big|\leq \frac{1}{m} \quad \text{ and } \quad  
|\nabla f^m_n(x)|\leq {\rm Lip}\big(f_n;B_{1/m}(x)\big)\; \text{ for all }x\in \R^d,
\end{equation}
whose existence is provided by Lemma \ref{lemma:approx_lip_via_smooth}.
Calling \(K_n\) the closed \(1\)-neighbourhood of \({\rm supp}(f_n)\), observe that 
\[|\nabla f^m_n|\leq {\rm Lip}(f_n)\chi_{K_n}\quad\text{ holds for every } m\in \N\]
and (by passing to the limsup in the second inequality in \eqref{eq:sequence_1}) that 
\[
\lims_m|\nabla f^m_n|\leq \lip_a(f_n).
\]
Thus, we may apply the reverse Fatou lemma and get that
\[
\lims_{m\to\infty}\int_{\R^d}|\nabla f^m_n|\,{\rm d}\mu\leq
\int_{\R^d}\lims_{m\to\infty}|\nabla f^m_n|\,{\rm d}\mu \leq
\int_{\R^d}\lip_a(f_n)\,{\rm d}\mu.
\]
Now pick 
\(m_n\in \N\) so that 
\[
\|f^{m_n}_n-f_n\|_{L^1_\mu(\R^d)}\leq \frac{1}{n}
\quad \text{ and }\quad \int_{\R^d}|\nabla f^{m_n}_n|\,{\rm d}\mu
\leq \int_{\R^d}\lip_a(f_n)\,{\rm d}\mu+\frac{1}{n}.
\]
By setting \(g_n\coloneqq f_n^{m_n}\in C^{\infty}_c(\R^d)\), 
we get (via a diagonalization argument) that 
\[
g_n\to f\,\text{ in }L^1_\mu(\R^d)
\quad\text{ and }\quad {\rm R}(f)
\leq \limi_{n\to\infty}\int_{\R^d}|\nabla g_n|\,{\rm d}\mu 
\leq \limi_{n\to\infty}\int_{\R^d}\lip_a(f_n)\,{\rm d}\mu.
\]
This gives that \({\rm R}(f)\leq |Df|(\R^d)\). Given that also 
the opposite inequality holds, 
by the fact that \(C^{\infty}_c(\R^d)\subseteq \LIP_c(\R^d)\), the proof of 
\eqref{eq:total_mass_smooth} is done.
\end{proof}
\subsection{Relation between vector fields and derivations}\label{ssec:Der_VF}
In this subsection we show that the space of bounded derivations 
with bounded divergence is isometrically isomorphic
to the space of bounded vector fields with bounded divergence. 
The main tool we are going to use is
the following
result that we refer to as the \emph{superposition principle for derivations}:
\begin{theorem}\label{thm:superposition_derivations}
Let \(\mathbf{b}\in {\rm Der}_b(\R^n,\mu)\) 
be such that \(|\mathbf b|, {\rm div}(\mathbf b)\in L^1_\mu(\R^d)\). 
Then there exists a finite, non-negative Borel measure $\pi$ on $C([0,1],\R^d)$
concentrated on non-constant absolutely continuous curves having constant speed and
such that
\begin{subequations}\begin{align}\label{eq:superpos_1}
\int_{\R^d}g\,\mathbf{b}(f)\,{\rm d}\mu&=\int\!\!\!\int_0^1 g(\gamma_t)\,
(f\circ\gamma)'_t\,
{\rm d}t\,{\rm d}\pi(\gamma)\quad\text{ for every }\, 
(g,f)\in\LIP(\R^d)\times \LIP_c(\R^d),\\\label{eq:superpos_2}
\int_{\R^d}g\,|\mathbf{b}|\,{\rm d}\mu&=\int\!\!\!\int_0^1 g(\gamma_t)\,|\dot\gamma_t|\,
{\rm d}t\,{\rm d}\pi(\gamma)\quad\text{ for every }\, g\in{\rm LIP}_c(\R^d).
\end{align}\end{subequations}
\end{theorem}

The above result is a consequence of a metric version 
(provided by Paolini and Stepanov in \cite{PS1, PS2}) of the
superposition principle for normal 1-currents
proven by Smirnov in \cite{S93} and of the fact that any element of 
\({\rm Der}_b(\R^d,\mu)\) induces a normal 1-current. 
Namely, for any \(\mathbf{b}\in {\rm Der}_b(\R^d,\mu)\), the map 
given by 
\[T_{\mathbf{b}}(g,f)\coloneqq\int g\,\mathbf{b}(f)\,{\rm d}\mu\quad\text{ for every }\,
(g,f)\in\LIP(\R^d)\times \LIP_c(\R^d),\]
defines a normal
$1$-current on $\R^d$. 
The formulation given in Theorem \ref{thm:superposition_derivations}
is due to \cite{GDMSP}.
\begin{theorem}\label{thm:derivations_vs_vectorfields}
The operator 
\(\Phi\colon {\rm D}_\infty({\rm div}_\mu)\to {\rm Der}_b(\R^d,\mu)\), given by
\[
\Phi(v)(f)\coloneqq v\cdot \nabla_\mu f\in L^{\infty}_\mu(\R^d),\quad
\text{ for every } v\in D_{\infty}({\rm div}_\mu)\text{ and every } f\in \LIP_c(\R^d),
\] is a 
bijection, \(\LIP_c(\R^d)\)-linear and satisfies 
\begin{equation}\label{eq:derivations_vs_vectorfields}
|\Phi(v)|=|v| \quad \text{ and }\quad {\rm div}\big(\Phi(v)\big)
={\rm div}_\mu(v)\quad \mu\text{-a.e.,}\,
\text{ for every }v\in {\rm D}_{\infty}({\rm div}_\mu).
\end{equation}
\end{theorem}
\begin{proof}\ \\
{\color{blue}\sc Step 1.}
First of all, given \(v\in {\rm D}_{\infty}({\rm div}_\mu)\),
we verify that \(\Phi(v)\in {\rm Der}_b(\R^d,\mu)\). The linearity of \(\Phi(v)\) clearly 
holds true, while the properties 1) and 2) in the Definition \ref{def:derivation} 
follow from the fact that the gradient operator 
\(\nabla_\mu\) satisfies the Leibniz rule (see Lemma \ref{lm:Leibniz_tan_gradient})
and from the \(\mu\)-a.e.\ inequality (granted by Proposition \ref{prop:Lipc_W1,1})
\[
\big|\Phi(v)(f)\big|\leq |v||\nabla_\mu f|\leq |v|\,\lip_a(f)\quad \text{ for every }f\in \LIP_c(\R^d),
\]
respectively. 
We now prove that \(|\Phi(v)|=|v|\in L^{\infty}_\mu(\R^d)\). 
Recalling formula \eqref{eq:formula_for_|b|}, we have that 
\[
\begin{split}
|\Phi(v)|=\, &\text{ess sup}\, \big\{|v\cdot\nabla_\mu f|:
\, f\in \LIP_c(\R^d),\,{\rm Lip}(f)\leq 1\big\}\\
\leq\, &\text{ess sup}\, \big\{|v|\,|\nabla_\mu f|:
\, f\in \LIP_c(\R^d),\,{\rm Lip}(f)\leq 1\big\} \leq |v|, 
\end{split}
\]
holds \(\mu\)-a.e..
To prove the opposite inequality, take a dense sequence
\((w_i)_i\subseteq \mathbb S^{d-1}\coloneqq \{w\in \R^d: |w|=1\}\).
Then for \(\mu\)-a.e.\ \(x\in \R^d\) we have that 
\(|v|(x)=\sup_{i\in \N}v(x)\cdot w_i\). Now, for every \(i, k\in \N\) choose 
\(f_{i,k}\in C^{\infty}_c(\R^d)\) such that
\(\nabla f_{i,k}=w_i\) on \(B_k(0)\). Then, given \(k\in \N\), we have that 
\[
\begin{split}
|v|(x)=\sup_{i\in \N}\,v(x)\cdot w_i= & \sup_{i\in \N}\,v(x)\cdot \nabla f_{i,k}(x)= 
\sup_{i\in \N}\,v(x)\cdot \nabla_\mu f_{i,k}(x)\\
=& \sup_{i\in \N}\, \Phi(v)(f_{i,k})(x)
\leq |\Phi(v)|(x)\,\lip_a(f_{i,k})(x)\leq |\Phi(v)|(x),
\end{split}
\]
for \(\mu\)-a.e.\ \(x\in B_k(0)\). By the arbitrariness of \(k\), we conclude that 
\(|v|\leq|\Phi(v)|\) holds \(\mu\)-a.e.\ in \(\R^d\). 

To see that  \(\Phi(v)\) admits bounded divergence, let us first observe that
for every \(f\in C^{\infty}_c(\R^d)\) it holds that
\begin{equation}\label{eq:phi(v)_div_smooth}
\int_{\R^d}\Phi(v)(f)\,{\rm d}\mu=\int_{\R^d}v\cdot \nabla_\mu f\,{\rm d}\mu
\overset{\eqref{eq:tangential_grad_smooth}}{=}\int_{\R^d}v\cdot\nabla f\,{\rm d}\mu
= -\int_{\R^d}f\,{\rm div}_{\mu}(v)\,{\rm d}\mu.
\end{equation}
Now, given any \(f\in \LIP_c(\R^d)\), we know that \(f\in W^{1,1}(\R^d,\mu)\) and 
thus we can find a sequence \((f_n)_n\subseteq C^{\infty}_c(\R^d)\) such that
\[f_n\to  f\;\text{ in } L^1_\mu(\R^d)\quad\text{ and }\quad
\nabla_\mu f_n\to
\nabla_\mu f\;\text{ in }L^1_\mu(\R^d;\R^d).\]
Thus, we can pass to the limit in \eqref{eq:phi(v)_div_smooth} and get that
\[
\int_{\R^d}\Phi(v)(f)\,{\rm d}\mu
=\lim_{n\to \infty}\int_{\R^d}v\cdot\nabla_\mu f_n\,{\rm d}\mu
=-\lim_{n\to \infty}\int_{\R^d}f_n\,{\rm div}_{\mu}(v)\,{\rm d}\mu=
-\int_{\R^d}f\,{\rm div}_{\mu}(v)\,{\rm d}\mu.
\]
By the arbitrariness of 
\(f\in \LIP_c(\R^d)\), this proves that \(\Phi(v)\) admits divergence 
and that \({\rm div}\big(\Phi(v)\big)={\rm div}_\mu (v)\in L^{\infty}_\mu(\R^d)\).

{\color{blue}\sc Step 2.} What remains to show is that \(\Phi\) is bijective.
The injectivity 
of \(\Phi\) is granted by the \(\mu\)-a.e.\ equality \(|\Phi(v)|=|v|\) proved in 
{\sc Step 1.}
Let us now fix \(\mathbf{b}\in {\rm Der}_b(\R^d,\mu)\) 
such that \(|\mathbf{b}|, {\rm div}(\mathbf b)\in L^1_\mu(\R^d)\).
Let \(\pi\) be the measure on the space of curves \(C\big([0,1],\R^d\big)\) given by the 
superposition principle in Theorem \ref{thm:superposition_derivations}. Define the map 
\({\sf D}\colon C\big([0,1],\R^d\big)\times [0,1]\to \R^d\times \R^d\) as
\[
{\sf D}(\gamma, t)\coloneqq (\gamma_t,\dot\gamma_t),\quad \text{ for every }
(\gamma, t)\in C\big([0,1],\R^d\big)\times [0,1].
\]
We further set \(\nu\coloneqq {\sf D}_*\big(\pi\otimes \mathcal L^1|_{[0,1]}\big)\).
Calling \(p\colon \R^d\times \R^d\to \R^d\) the canonical projection map, i.e.\ 
\(p(x,v)=x\) for every \((x,v)\in \R^d\times \R^d\), we disintegrate the measure 
\(\nu\) with respect to the map \(p\), getting a measurable family  
\(\{\nu_x\}_{x\in \R^d}\) of probability measures
\(\nu_x\) on \(\R^d\) satisfying
\[
\int_{\R^d\times \R^d}g(x,v)\,{\rm d}\nu(x,v)=\int_{\R^d}\int_{\R^d}g(x,\cdot)\,
{\rm d}\nu_x\,{\rm d}p_*\nu(x),\quad \text{ for every }g\in L^1_\nu(\R^d\times \R^d).
\]
We claim that 
\(p_*\nu\ll \mu\). Indeed, by using \eqref{eq:superpos_2} we have for every 
\(g\in\LIP_c(\R^d)\) that
\[\begin{split}
\int_{\R^d}g|\mathbf{b}|\,{\rm d}\mu&=
\int\!\!\!\int_0^1 g(\gamma_t)\,|\dot{\gamma_t}|\,{\rm d} t\,{\rm d}\pi(\gamma)\\
&=\int_{\R^d} g(x)\bigg(\int_{\R^d}|w|\,{\rm d}\nu_x(w)\bigg){\rm d} p_{*}\nu(x).
\end{split}\]
By the arbitrariness of \(g\in \LIP_c(\R^d)\), we have that 
\begin{equation}\label{eq:p*nu_abs_cont_mu}
|\mathbf{b}|\mu=\int_{\R^d}|w|\,{\rm d}\nu_{(\cdot)}(w)\,p_*\nu.
\end{equation} 
Since the measure \(\pi\) is concentrated on non-constant 
curves having constant speed, we have that \(\dot{\gamma_t}\neq 0\) 
for \((\pi\otimes \mathcal L^1|_{[0,1]})\)-a.e.\ \((\gamma,t)\). This implies that
\(\int_{\R^d}|w|\,{\rm d}\nu_x(w)>0\) holds for \(p_*\nu\)-a.e.\ \(x\in \R^d\). 
Therefore, 
we conclude that \(p_*\nu\ll\mu\).

{\color{blue}\textsc{Step 3.}} Now, we define 
\[
v(x)\coloneqq \frac{{\rm d}p_*\nu}{{\rm d}\mu}(x)\int_{\R^d}w\,{\rm d}\nu_x(w),
\quad \text{ for }\mu\text{-a.e.\ }x\in \R^d.
\] 
Our aim is to show that \(v\in {\rm D}_{\infty}({\rm div}_\mu)\) and that 
\(\mathbf{b}(f)=\nabla_\mu f\cdot v=\Phi(v)(f)\) for every \(f\in \LIP_c(\R^d)\).
First of all, observe that the formula \eqref{eq:p*nu_abs_cont_mu} ensures that 
\(|v|\leq |\mathbf{b}|\) holds \(\mu\)-a.e., thus \(v\in L_\mu^{\infty}(\R^d)\).
By using formula \eqref{eq:superpos_1} and by unwrapping the above definitions 
we have the following:
given any \(g\in{\rm LIP}_c(\R^d)\) and \(f\in C^\infty_c(\R^d)\) it holds that
\[\begin{split}
\int_{\R^d} g\,\mathbf{b}(f)\,{\rm d}\mu&=
\int\!\!\!\int_0^1 g(\gamma_t)\,\nabla f(\gamma_t)\cdot\dot\gamma_t\,{\rm d} t\,
{\rm d}\pi(\gamma)\\
&=\int_{\R^d} g(x)\bigg(\int\nabla f(x)\cdot w\,{\rm d}\nu_x(w)\bigg){\rm d }p_*\nu(x)\\
&=\int_{\R^d} g(x)\bigg(\frac{{\rm d} p_*\nu}{{\rm d}\mu}(x)
\int_{\R^d}\nabla f(x)\cdot w\,{\rm d}\nu_x(w)\bigg){\rm d}\mu(x)\\
&=\int_{\R^d} g(x)\,\nabla f(x)\cdot v(x)\,{\rm d}\mu(x).
\end{split}\]
Thus, since \(g\in \LIP_c(\R^d)\) was arbitrary, we deduce that
\(\mathbf{b}(f)(x)=\nabla f(x)\cdot v(x)\) for \(\mu\)-a.e.\ \(x\in \R^d\).
This also ensures that 
\[
\int_{\R^d}\nabla f\cdot v\,{\rm d}\mu=\int_{\R^d}\mathbf{b}(f)\,{\rm d}\mu
=-\int_{\R^d}{\rm div}(\mathbf{b})\,f\,{\rm d}\mu\]
holds for every \(f\in C^{\infty}_c(\R^d)\). Hence, \(v\) is a vector field with divergence 
and \({\rm div}_\mu(v)={\rm div}(\mathbf{b})\). 
All in all, we have proved that \(v\in {\rm D}_{\infty}({\rm div}_\mu)\). 
Consequently, we have that 
\begin{equation}\label{eq:b(f)_smooth}
\mathbf{b}(f)=\nabla f\cdot v=\nabla_\mu f\cdot v\quad \text{ holds for every }\,
f\in C^\infty_c(\R^d).
\end{equation}
By approximation, we can obtain \eqref{eq:b(f)_smooth} for every \(f\in \LIP_c(\R^d)\),
proving that 
\(\Phi\big({\rm D}_{\infty}({\rm div}_\mu)\big)\subseteq
\{\mathbf{b}\in {\rm Der}_b(\R^d,\mu):\,
|\mathbf b|, {\rm div}(\mathbf b)\in L^1_\mu(\R^d)\}\eqqcolon \mathcal D.\)\\

{\color{blue}\textsc{Step 4.}} It remains to show that \(\Phi\) is surjective.
Fix \(\mathbf{b}\in {\rm Der}_b(\R^d, \mu)\) and fix a sequence 
\((\eta_n)_n\subseteq C^\infty_c(\R^d)\) such that \(0\leq \eta_n\leq 1\), \(\eta_n=1\) 
on \(B_n(0)\) and \({\rm Lip}(\eta_n)=1\) for each \(n\in \N\). We set 
\(\mathbf b_n\coloneqq \eta_n\mathbf b\) and note that \(\mathbf b_n\in \mathcal D\). 
Moreover, 
\(|\mathbf b_n|\leq |\mathbf b|\) and 
\(|{\rm div}(\mathbf b_n)|=|\mathbf b(\eta_n)+\eta_n{\rm div}(\mathbf b)|
\leq |\mathbf b|+|{\rm div}(\mathbf b)|\). By \textsc{Step 3} for every \(n\in \N\)
we have the existence of an element \(v_n\in {\rm D}_{\infty}({\rm div}_\mu)\) such 
that \(\Phi(v_n)=\mathbf b_n\). Also, \(|v_n|=|\mathbf b_n|\leq |\mathbf b|\) and 
\(|{\rm div}_\mu(v_n)|=|{\rm div}(\mathbf b_n)|\leq |\mathbf b|+|{\rm div}(\mathbf b)|\), 
thus (up to a subsequence) we have that \(v_n\rightharpoonup v\) weakly\(^*\) in 
\(L^\infty_\mu(\R^d; \R^d)\) for some \(v\in L^\infty_\mu(\R^d; \R^d)\)
and \({\rm div}_\mu(v_n)\rightharpoonup h\) weakly\(^*\) in 
\(L^\infty_\mu(\R^d)\) for some \(h\in L^\infty_\mu(\R^d)\). Moreover, due to the 
closure of the operator \({\rm div}_\mu\) we have that 
\(v\in {\rm D}_\infty({\rm div}_\mu)\)
and that \(h={\rm div}_\mu(v)\). We only need to check that \(\Phi(v)=\mathbf b\). 
Let us fix \(f\in \LIP_c(\R^d)\). Then, for every \(n\in \N\) we have that 
\(\mathbf b_n(f)=\Phi(v_n)(f)=v_n\cdot \nabla_\mu f\). 
Since \(v_n\cdot \nabla_\mu f\) and \(\mathbf b_n(f)=\eta_n\mathbf b(f)\) converge
weakly\(^*\) in \(L^\infty_\mu(\R^d)\) to \(v\cdot \nabla_\mu f\) and \(\mathbf b(f)\),
respectively, we get that \(\mathbf b(f)=v\cdot \nabla_\mu f=\Phi(v)(f)\). By the 
arbitrariness of \(f\in \LIP_c(\R^d)\), we get the surjectivity of \(\Phi\) and conclude 
the proof.
\end{proof}
\subsection{Equivalence of BV spaces}
As our main result, we have the following equivalent characterizations of 
the BV space: 

\begin{theorem}[Equivalent characterizations of BV function]
\label{thm:equiv_BV}
It holds that
\[
{\rm BV}(\R^d,\mu)=
{\rm BV}_{\rm Der}(\R^d,\mu)={\rm BV}_{\rm Lip}(\R^d,\mu)
={\rm BV}_{C^\infty}(\R^d,\mu).
\]
Moreover, it holds that
\(|D_\mu f|=|D_\mu f|_{\rm Der}=|D_\mu f|_{\rm Lip}=|D_\mu f|_{C^\infty}\)
for every \(f\in{\rm BV}(\R^d,\mu)\).
\end{theorem}
\begin{proof}\ \\
{\color{blue}\textsc{Step 1.}}
First of all, the fact that
\({\rm BV}_{\rm Der}(\R^d,\mu)={\rm BV}_{\rm Lip}(\R^d,\mu)={\rm BV}_{C^\infty}(\R^d,\mu)\)
and \(|D_\mu f|_{\rm Der}=|D_\mu f|_{\rm Lip}=|D_\mu f|_{C^\infty}\) for every
\(f\in{\rm BV}_{\rm Der}(\R^d,\mu)\) follows from \cite[Theorem 4.5.3]{DMPhD} and 
Theorem \ref{thm:Lip_Cinfty}.
Moreover, it follows from Theorem \ref{thm:derivations_vs_vectorfields} that 
\(|D_\mu f|(\R^d)=|D_\mu f|_{\rm Der}(\R^d)\) for every 
\(f\in {\rm BV}_{\rm Der}(\R^d,\mu)\), which implies 
\({\rm BV}_{\rm Der}(\R^d,\mu)\subseteq {\rm BV}(\R^d,\mu)\).

For the purposes of the next step, we recall that
it was proved in \cite[Theorem 5.1]{BBF} that a given
function \(f\in L^1_\mu(\R^d)\) belongs to \({\rm BV}(\R^d,\mu)\)
if and only if
\begin{equation}\label{eq:yet_another_BV}
\exists(f_n)_{n\in\N}\subseteq C^\infty_c(\R^d):\quad
f_n\to f\,\text{ in }L^1_\mu(\R^d),\quad\sup_{n\in\N}\int_{\R^d}
|\nabla_\mu f_n|\,\d\mu<+\infty.
\end{equation}

{\color{blue}\textsc{Step 2.}} Next we claim that
\begin{equation}\label{eq:equiv_BV_aux2}
{\rm BV}(\R^d,\mu)\subseteq{\rm BV}_{\rm Der}(\R^d,\mu).
\end{equation}
In order to prove it, fix any \(f\in{\rm BV}(\R^d,\mu)\).
Define \(\mathcal L_f\colon{\rm Der}_b(\R^d,\mu)\to 
\big(\LIP_c(\R^d),\|\cdot\|_{C_b(\R^d)}\big)^*\) as
\begin{equation}\label{eq:def_mathcalL_f(b)}
\mathcal L_f(\mathbf b)(h)\coloneqq-\int_{\R^d} f\,{\rm div}(h\mathbf b)
\,{\rm d}\mu,\quad \text{ for every }\mathbf b\in{\rm Der}_b(\R^d,\mu)
\text{ and }h\in \LIP_c(\R^d).
\end{equation}
To see that \(\mathcal L_f(\mathbf b)\) is indeed an element of 
\(\big(\LIP_c(\R^d),\|\cdot\|_{C_b(\R^d)}\big)^*\), note that
for any \(\mathbf b\in {\rm Der}_b(\R^d,\mu)\) and
\(h\in\LIP_c(\R^d)\) with \(h\neq 0\) we may compute
\[
\mathcal L_f(\mathbf b)(h)=
-\int_{\R^d}f\,{\rm div}\left(\frac{h\|h\|_{C_b(\R^d)}}
{\|h\|_{C_b(\R^d)}}\mathbf b\right)\,{\rm d}\mu=
-\|h\|_{C_b(\R^d)}\int_{\R^d}f\,{\rm div}\left(\frac{h}
{\|h\|_{C_b(\R^d)}}\mathbf b\right)\,{\rm d}\mu.
\]
We now apply Theorem \ref{thm:derivations_vs_vectorfields}: since
\(v\coloneqq\frac{h}{\|h\|_{C_b(\R^d)}}\Phi^{-1}(\mathbf b)\)
is a competitor in the definition of \(\|D_\mu f\|\) and
\({\rm div}_\mu(v)={\rm div}\big(\frac{h}{\|h\|_{C_b(\R^d)}}\boldsymbol b\big)\),
we get \(|\mathcal L_f(\mathbf b)|\leq \|h\|_{C_b(\R^d)}\|D_\mu f\|\),
thus the continuity of 
\(\mathcal L_f(\mathbf b)\). The linearity is clear from the very
definition of \(\mathcal L_f(\mathbf b)\), thus we conclude that
\(\mathcal L_f(\mathbf b)\) is an element of  
\(\big(\LIP_c(\R^d),\|\cdot\|_{C_b(\R^d)}\big)^*\). Being
\(\LIP_c(\R^d)\) dense in \(C_0(\R^d)\) with respect to the
\(C_b(\R^d)\)-norm, we can uniquely extend \(\mathcal L_f(\mathbf b)\)
to an element of \(\big(C_0(\R^d),\|\cdot\|_{C_b(\R^d)}\big)^*\),
which we still call \(\mathcal L_f(\mathbf b)\). Given that
\(\big(C_0(\R^d),\|\cdot\|_{C_b(\R^d)}\big)^*\) can be identified
with \(\mathscr M(\R^d)\) (recall Remark \ref{rmk:Measures}), there exists 
a unique \(L_f(\mathbf b)\in\mathscr M(\R^d)\) such that
\begin{equation}\label{eq:def_L_f(b)}
{\mathcal L}_f(\mathbf b)(h)=\int_{\R^d} h\,{\rm d}L_f(\mathbf b),
\quad \text{ for every }h\in C_0(\R^d).
\end{equation}
To verify that \(L_f(v)(\R^d)=-\int_{\R^d} f\,{\rm div}_\mu(v)\,{\rm d}\mu,\) 
pick a sequence \((h_n)_{n\in\N}\subseteq \LIP_c(\R^d)\) of
\(1\)-Lipschitz functions \(h_n\colon\R^d\to[0,1]\) such that
\(h_n=1\) on \(B_n(0)\) for every \(n\in \N\). In particular,
\(h_n(x)\to 1\) and \(\lip_a(h_n)(x)\to 0\) for every \(x\in\R^d\).
Since for any \(n\in\N\) we have the \(\mu\)-a.e.\ inequality
\[
\big|f\big(h_n{\rm div}(\mathbf b)+\mathbf b(h_n)\big)\big|
\leq|f|\big(|{\rm div}(\mathbf b)|+|\mathbf b|\,\lip_a(h_n)\big)
\leq|f|\big(|{\rm div}(\mathbf b)|+|\mathbf b|\big)\in L^1_\mu(\R^d),
\]
by using twice the dominated convergence theorem we deduce that
\[\begin{split}
\int_{\R^d}\d L_f(\mathbf b)
&\overset{\phantom{\eqref{eq:Leibniz_div}}}=
\lim_{n\to\infty}\int_{\R^d}h_n\,\d L_f(\mathbf b)
\overset{\eqref{eq:def_L_f(b)}}=
\lim_{n\to\infty}\mathcal L_f(\mathbf b)(h_n)
\overset{\eqref{eq:def_mathcalL_f(b)}}=
-\lim_{n\to\infty}\int_{\R^d}f\,{\rm div}(h_n\mathbf b)\,\d\mu\\
&\overset{\eqref{eq:Leibniz_div}}=
-\lim_{n\to\infty}\int_{\R^d}f\big(h_n{\rm div}(\mathbf b)
+\mathbf b(h_n)\big)\,\d\mu=-\int_{\R^d}f\,{\rm div}(\mathbf b)\,\d\mu.
\end{split}\]
The linearity of the map 
\({\rm Der}_b(\R^d,\mu)\ni \mathbf b\mapsto L_f(\mathbf b)\in
\mathscr M(\R^d)\) is clear from the very definition.
Moreover, given any \(g,h\in\LIP_c(\R^d)\), we can compute
\[
\int_{\R^d}h\,\d L_f(g\mathbf b)=\mathcal L_f(g\mathbf b)(h)=
-\int_{\R^d}f\,{\rm div}(hg\mathbf b)\,\d\mu=
\mathcal L_f(\mathbf b)(hg)=\int_{\R^d}hg\,\d L_f(\mathbf b),
\]
which, thanks to the arbitrariness of \(h\in\LIP_c(\R^d)\),
implies \(L_f(g\mathbf b)=g L_f(\mathbf b)\) for all
\(g\in\LIP_c(\R^d)\). Hence, we have proved that the operator
\(L_f\) is \(\LIP_c(\R^d)\)-linear.
We are just left to prove the continuity of \(L_f\) with respect
to the \(\|\cdot\|_b\)-norm on \({\rm Der}_b(\R^d,\mu)\).
By applying \eqref{eq:yet_another_BV}, we can find a sequence
\((f_n)_{n\in\N}\subseteq C^\infty_c(\R^d)\) and \(C\geq 0\)
such that \(f_n\to f\) in \(L^1_\mu(\R^d)\) and
\(\int_{\R^d}|\nabla_\mu f_n|\,\d\mu\to C\). 
For any \(n\in\N\) and \(h\in\LIP_c(\R^d)\), we can estimate
\[\begin{split}
\left|\int f_n \,{\rm div}(h\mathbf b)\,{\rm d}\mu\right|
&\overset{\eqref{eq:derivations_vs_vectorfields}}=\left|
\int f_n\,{\rm div}_\mu\big(h\,\Phi^{-1}(\mathbf b)\big)\,\d\mu\right|
=\left|\int_{\R^d}h\,\nabla_\mu f_n\cdot\Phi^{-1}(\mathbf b)
\,\d\mu\right|\\
&\overset{\phantom{\eqref{eq:derivations_vs_vectorfields}}}\leq
\|h\|_{C_b(\R^d)}\|\Phi^{-1}(\mathbf b)\|_{L^\infty_\mu(\R^d;\R^d)}
\int_{\R^d}|\nabla_\mu f_n|\,\d\mu\\
&\overset{\eqref{eq:derivations_vs_vectorfields}}=
\|h\|_{C_b(\R^d)}\|\mathbf b\|_b\int_{\R^d}|\nabla_\mu f_n|\,\d\mu,
\end{split}\]
whence by letting \(n\to\infty\) it follows
\(|\mathcal L_f(\mathbf b)(h)|=\lim_n
\big|\int f_n \,{\rm div}(h\mathbf b)\,{\rm d}\mu\big|
\leq\|h\|_{C_b(\R^d)}\|\mathbf b\|_b C\), thus
\[
\|L_f(\mathbf b)\|_{\sf TV}=
\sup_{\substack{h\in \LIP_c(\R^d):\\  \|h\|_{C_b(\R^d)}\leq 1}}
|{\mathcal L}_f(\mathbf b)(h)|\leq C\|\mathbf b\|_b.
\]
This yields continuity of the map \({\rm Der}_b(\R^d,\mu)
\ni\mathbf b\mapsto L_f(\mathbf b)\in\mathscr M(\R^d)\).
All in all, we have shown that \(f\in{\rm BV}_{\rm Der}(\R^d,\mu)\)
with \(Df=L_f\), thus accordingly the claim \eqref{eq:equiv_BV_aux2}
is proved.

{\color{blue}\textsc{Step 3.}} 
So far, we have shown that
\({\rm BV}(\R^d,\mu)=
{\rm BV}_{\rm Der}(\R^d,\mu)={\rm BV}_{\rm Lip}(\R^d,\mu)
={\rm BV}_{C^\infty}(\R^d,\mu)\). 
To conclude, fix a function
\(f\in{\rm BV}(\R^d,\mu)\) and an open set \(\Omega\subseteq\R^d\).
The fact that 
\(|D_\mu f|_{C^\infty}(\Omega)
=|D_\mu f|_{\rm Lip}(\Omega)
=|D_\mu f|_{\rm Der}(\Omega)\)
is granted by Theorem \ref{thm:Lip_Cinfty} and \cite[Theorem 4.5.3]{DMPhD}.
Moreover, it readily follows from Theorem
\ref{thm:derivations_vs_vectorfields} that
\(|D_\mu f|(\Omega)=|D_\mu f|_{\rm Der}(\Omega)\) as well.
This is sufficient to conclude that
\(|D_\mu f|=|D_\mu f|_{\rm Der}=|D_\mu f|_{\rm Lip}=|D_\mu f|_{C^\infty}\) as measures,
thus completing the proof of the statement.
\end{proof}
\subsection{Relation between \texorpdfstring{\(W^{1,1}\)}{W11} spaces}
Aim of this brief section is to investigate the relation between
the Sobolev spaces \(W^{1,1}(\R^d,\mu)\) and \(W^{1,1}_{\rm Lip}(\R^d,\mu)\). 
\begin{theorem}[Relation between \(W^{1,1}\) spaces]
\label{thm:relation_W11}
It holds that
\[
W^{1,1}_{\rm Lip}(\R^d,\mu)\subseteq W^{1,1}(\R^d,\mu).
\]
Moreover, it holds that \(|\nabla_\mu f|\leq|\nabla f|_{rs}\)
\(\mu\)-a.e.\ for every \(f\in W^{1,1}_{\rm Lip}(\R^d,\mu)\).
\end{theorem}
\begin{proof}
To prove the statement amounts to showing that
\begin{equation}\label{eq:relation_W11_cl}
f\in W^{1,1}_{\rm Lip}(\R^d,\mu)\quad\Longrightarrow\quad
f\in W^{1,1}(\R^d,\mu)\,\text{ and }\,|\nabla_\mu f|\leq|\nabla f|_{rs}
\,\text{ in the }\mu\text{-a.e.\ sense.}
\end{equation}
Taking Remark \ref{rmk:approx_lip_a_strong} into account,
we can find a sequence \((f_n)_{n\in\N}\subseteq\LIP_c(\R^d)\)
and a non-negative function \(H\in L^1_\mu(\R^d)\) 
such that \(f_n\to f\) strongly in \(L^1_\mu(\R^d)\),
\(\lip_a(f_n)\rightharpoonup|\nabla f|_{rs}\) weakly in
\(L^1_\mu(\R^d)\), and \(\lip_a(f_n)\leq H\) \(\mu\)-a.e.\ for
every \(n\in\N\). Given any \(n\in\N\),
define \(f_n^k\coloneqq\rho_{1/k}*f_n\in C^\infty_c(\R^d)\)
for every \(k\in\N\). Lemma \ref{lemma:approx_lip_via_smooth}
says that \(f_n^k\to f_n\) strongly in \(L^1_\mu(\R^d)\) as
\(k\to\infty\). Lemma \ref{lem:conv_to_lip_a} gives
\[
|\nabla_\mu f_n^k|\leq|\nabla f_n^k|
\overset{\eqref{eq:approx_lip_via_smooth_2}}\leq
{\rm Lip}\big(f_n;B_{2/k}(\cdot)\big)\to\lip_a(f_n),
\quad\text{ strongly in }L^1_\mu(\R^d)\text{ as }k\to\infty.
\]
Then Proposition \ref{prop:FA_facts} yields the existence of a function
\(G_n\in L^1_\mu(\R^d)\) such that (up to a subsequence in \(k\))
it holds \(G_n\leq\lip_a(f_n)\leq H\) \(\mu\)-a.e.\ and
\(|\nabla_\mu f_n^k|\rightharpoonup G_n\) weakly in
\(L^1_\mu(\R^d)\) as \(k\to\infty\). By applying Proposition
\ref{prop:FA_facts} again, we can also find a function
\(G\in L^1_\mu(\R^d)\) such that \(G\leq|\nabla f|_{rs}\)
\(\mu\)-a.e.\ and (up to a subsequence in \(n\))
\(G_n\rightharpoonup G\) weakly in \(L^1_\mu(\R^d)\).
Thanks to a diagonalization argument, we can construct a
sequence \((k(n))_{n\in\N}\subseteq\N\) such that the
functions \(g_n\coloneqq f_n^{k(n)}\in C^\infty_c(\R^d)\)
satisfy \(g_n\to f\) strongly in \(L^1_\mu(\R^d)\) and
\(|\nabla_\mu g_n|\rightharpoonup G\) weakly in \(L^1_\mu(\R^d)\).
This implies that \(|\nabla f|_{rs}\in{\rm TRS}(f)\), whence
(by Theorem \ref{thm:equiv_char_W11}) it follows that
\(f\in W^{1,1}(\R^d,\mu)\) and \(|\nabla_\mu f|\leq|\nabla f|_{rs}\)
\(\mu\)-a.e., getting \eqref{eq:relation_W11_cl}. Therefore,
the statement is achieved.
\end{proof}

\bibliographystyle{siam}

\end{document}